\documentclass[preprint,12pt]{elsarticle}
\usepackage[english]{babel}
\usepackage{listings}
\usepackage{float}
\usepackage{epsf}
\usepackage{graphicx} 
\usepackage[nottoc, notlof, notlot]{tocbibind} 
\usepackage{subfigure}
\usepackage{multirow}
\usepackage{tikz}
\usepackage{algorithm}
\usepackage{algorithmic}

 \usepackage{graphics}
\usepackage{amsmath,amsfonts,amssymb} 
\newtheorem{theorem}{Theorem}[section]

\newtheorem{remark}{Remark}
\newenvironment{proof}[1][Proof]{\begin{trivlist}
\item[\hskip \labelsep {\bfseries #1}]}{\end{trivlist}}
\newcommand{\argmin}{\operatornamewithlimits{argmin}}
 \usepackage{lineno}
\graphicspath{{./}}

\journal{Journal of Computational Physics}

\begin{document}

\begin{frontmatter}



\title{Numerical simulation of a class of models that combine several mechanisms of dissipation: fracture, plasticity, viscous dissipation.}

\author[label1]{Eric Bonnetier}
\author[label1]{Luk\'a\v{s} Jakab\v{c}in}
\author[label1]{St\'ephane Labb\'e}
\author[label2]{Anne Replumaz}
\address[label1]{Laboratoire Jean Kuntzmann, Universit\'e de Grenoble-Alpes, CNRS, Grenoble, France}
\address[label2]{ISTerre, Universit\'e de Grenoble-Alpes, CNRS, Grenoble, France}

\author{}

\address{}
\begin{abstract}
We study a class of time evolution models that contain  dissipation mechanisms exhibited by geophysical materials during deformation: 
plasticity, viscous dissipation and fracture.
We formally prove that they satisfy a Clausius-Duhem  type inequality.
We describe a semi-discrete time evolution associated with these models,
and report numerical 1D and 2D traction experiments, that illustrate that several dissipation regimes
can indeed take place during the deformation. 
Finally, we report 2D numerical simulation of an experiment by Peltzer and Tapponnier, 
who studied the indentation of a layer of plasticine as an analogue model for geological materials.
\end{abstract}



\end{frontmatter}



\section{Introduction}

In this paper, we study a class of models that combine several mechanisms of dissipation: 
plasticity, visco-plasticity, visco-elasticity and fracture. 

Our goal is to investigate whether models from solid mechanics could be pertinent to describe geophysical materials
(and particularly the lithosphere on continental scales), as advocated by Peltzer and Tapponnier \cite{PeltzerTapponnier}, while others (see for example~\cite{England1}, \cite{England2})
prefer descriptions based on fluid mechanics.
The solid mechanics approach would have advantage to account for cracks in the formation of geological faults. They illustrate their claim  with analogue experiments, where a rigid indenter deforms a layer of plasticine, to model the action of the Indian sub-continent on
the Tibetan plateau.
The plasticine experiments seems to reproduce the geophysical scenario of creation of the Asian faults and the extrusion of the South-Asian block (including Vietnam).

No general consensus prevails on the modeling of crack initiation and propagation, even in homogeneous materials. 
The popular Griffith model, much in use in the engineering community, suffers from various shortcomings. 
For instance it does not account for crack nucleation, and assumes pre-determined crack paths. 
In the last decade, a series of investigation initiated by Francfort and Marigo \cite{FrancfortMarigo} 
has addressed the mathematical foundation of fracture mechanics, using new concepts that have emerged 
from the mathematical modeling of composite materials and from the calculus of variation. 
This approach postulates that crack evolution is governed by the minimization of a total energy, among all 
possible crack states. 

Our models of fracture are inspired by this work, though we only consider fracture via a phase-field approximation.
In other words, the geometry of possible cracks is captured by a function $v$ with values between $0$ and $1$,
$v=1$ in the healthy parts that do not contain cracks. The length of the cracks, a quantity that contributes to the
total energy, is approximated via a functionnal introduced by Ambrosio and Tortorelli~\cite{Ambrosio} and Bourdin~\cite{Bourdin}.
The numerical simulations of fracture in a purely elastic medium, using such phase-field approximation,
was carried out by Bourdin, Francfort and Marigo~\cite{Bourdin2}, \cite{BourdinFrancfortMarigo1}, \cite{BourdinFrancfortMarigo}.
A model combining elasticity, visco-elasticity and fracture regularized via phase-field, is analyzed in~\cite{Larsen},
where the main point is how to define a consistent evolution  as the limit of semi-discrete approximations
in time.
Our class of models extends this work to the case when plastic behavior and viscoplastic behavior can occur.

From a thermodynamical point of view, we interpret the phase field function $v$, that tracks the location and 
propagation of cracks, not only as a variable for numerical approximation,
but as a global thermodynamical internal variable. 
We show that our models are consistent with thermodynamics, in the 
sense that they satisfy a Clausius-Duhem type inequality. 
We propose a numerical scheme for a space-time discretization of the evolution,
and analyze its advantages and shortcomings on 1D et 2D traction experiments and on the experiment by Peltzer and Tapponnier~\cite{PeltzerTapponnier}.

The paper is organized as follows. In Section 2, we describe the proposed models 
with regularized fracture and define their evolution in time.
Section 3 is dedicated to showing that they satisfy a Clausius-Duhem type inequality.
In Section 4, we introduce a semi-discrete time evolution, which is the base, in the final section, for numerical experiments
in the case of 1D, 2D traction and 2D plasticine experiment.
In particular we show that several dissipation mechanisms can be expressed according to the choice of parameters.

\section{Description of models with several dissipation mechanisms.}


\subsection{Notations.}

Throughout the paper, $\Omega$ denotes a bounded connected open set in $\mathbb{R}^{2}$ with Lipschitz boundary 
$\partial\Omega=\partial\Omega_D\cup\partial\Omega_N$, where $\partial\Omega_D,\, \partial\Omega_N$ are 
disjoint measurable sets. We denote time derivatives with a dot and $\argmin_{v\in V}\mathcal{F}(v)$ denotes a function $u$ that minimizes
$\mathcal{F}$ over $V$.

Given $T_f>0$, we denote by $L^p((0,T_f),X)$, $W^{k,p}((0,T_f),X)$, the Lebesgue and Sobolev spaces involving time [see \cite{Evans} p. 285], 
where X is a Banach space. The set of symmetric $2 \times 2$ matrices is denoted by $\mathbb{M}_{\text{sym}}^{2\times2}$ .
 For $\xi,\zeta \in \mathbb{M}_{\text{sym}}^{2\times2}$ we define the scalar product between matrices
$\zeta:\xi:=\sum_{ij}\zeta_{ij}\xi_{ij}$, and the associated matrix norm by $|\xi|:=\sqrt{\xi:\xi}$.
Let A be the fourth order tensor of Lam\'e coefficients and B a suitable symmetric-fourth order tensor. 
We assume that for some constants  $0<\alpha_1 \leq\ \alpha_2<\infty$, they satisfy the ellipticity conditions
\begin{eqnarray*}
\forall\; e \in \mathbb{M}_{\text{sym}}^{2\times2},\quad \alpha_1|e|^{2}\leq Ae:e\leq \alpha_2|e|^{2}\quad\text{and}\quad\alpha_1|e|^{2}\leq Be:e\leq \alpha_2|e|^{2}
\end{eqnarray*}
The mechanical unknowns of our model are the displacement field $ u: \Omega \times[0,T_f]\rightarrow \mathbb{R}^2$,
the elastic strain $e: \Omega \times[0,T_f]\rightarrow \mathbb{M}_{\text{sym}}^{2\times2}$, the plastic strain $p: \Omega \times[0,T_f]\rightarrow \mathbb{M}_{\text{sym}}^{2\times2}$. 
We assume $u$ and $\nabla u$ remain small. So that the relation between the deformation tensor $E$ and the displacement field is given by
\begin{eqnarray*}
Eu := \frac{1}{2}(\nabla u + \nabla u^T).
\end{eqnarray*}
We also assume that $Eu$ decomposes as an elastic part and
a plastic part
\begin{eqnarray*}
Eu &=& e + p.
\end{eqnarray*}
For $w \in H^1(0,T_f, H^1(\Omega, \mathbb{R}^2))$, which represents an applied boundary displacement, we define for $t\in[0,T_f]$ the set of kinematically admissible fields by
\begin{eqnarray*}
A_{adm}(w(t)) &:=& \{(u,e,p) \in H^{1}(\Omega,\mathbb{R}^{2}) \times L^{2}(\Omega,\mathbb{M}_{\text{sym}}^{2\times2})\times L^{2}(\Omega,\mathbb{M}_{\text{sym}}^{2\times2})~:  \\
&& Eu= e+p\, \quad a.e.\,\, \text{in} \,\, \Omega,\;u=w(t) \quad a.e.\;\; \text{on}  \; \; \partial \Omega_D\}.
\end{eqnarray*}
For $f \in C^1([0,T_f], L^2(\Omega)^2)$, and $g \in C^1([0,T_f],H^{-1/2}(\Omega)^2)$,
the external forces at time $t \in [0,T_f]$ are collected into
$$
\langle l(t), u \rangle	:= \int_\Omega f(t).u\,dx+\int_{\partial\Omega_N} g(t).u\,ds.
$$
For a fixed constant $\tau>0$, we define $\mathbb{K}:=\{q\in\mathbb{M}_{\text{sym}}^{2\times2};\, |q|\leq \tau\quad a.e.\, \text{in} \, \Omega\}$. 
We define $H:\mathbb{M}_{\text{sym}}^{2\times2}\rightarrow[0,\infty]$ the support function  of $\mathbb{K}$ by
$$
H(p):=\sup_{\theta\in \mathbb{K}}\,\, \theta:p=\tau|p|,
$$
and a perturbed dissipation potential $H_{\beta}$ by
$$
H_{\beta}(p):=H(p)+\dfrac{\beta}{2}|p|^2,
$$
where $\beta>0$ plays the role of a regularization parameter.
The variational approach to fracture \cite{FrancfortMarigo}, \cite{BourdinFrancfortMarigo} 
is based on Griffith's idea that the crack growth and crack path are determined by the competition between
the elastic energy release, when the crack increases, and the energy
dissipated to create a new crack. 
We approximate the fracture (see Figure~\ref{rupture}) by a phase field function 
$v: \Omega\times[0,T_f]\rightarrow [0,1]$ that depends on two parameters:
\begin{itemize}
\item $\epsilon>0$, the parameter of space regularization, relates to the width of the generalized fracture,
\item $\eta>0$ is a parameter, that preserves the ellipticity of the elastic energy. In \cite{Ambrosio}, $\eta$ scales as $o(\varepsilon)$ as $\varepsilon\rightarrow0$
in the approximation of a true crack by a phase-field function.
\end{itemize}
The nucleation and propagation of cracks, and the material deformation 
result from minimizing at each time a global energy, that contains several terms:
\begin{eqnarray*}
 \mathcal{E}_{total}:= \mathcal{E}_{el}+ \mathcal{E}_{p}+ \mathcal{E}_{h}+ \mathcal{E}_{ve}+ \mathcal{E}_{vp}+ \mathcal{E}_{S}.
\end{eqnarray*}
The elastic energy is defined as
\begin{align*}
&\mathcal{E}_{el}\,:\,  L^{2}(\Omega,\mathbb{M}_{\text{sym}}^{2\times2}) \times H^1(\Omega,\mathbb{R})\rightarrow \mathbb{R}\\
&(e,v)\longmapsto \mathcal{E}_{el}(e,v) = \dfrac{1}{2} \int_{\Omega}\left( v^{2}+\eta\right)Ae:e \,dx. 
\end{align*}
The plastic dissipated energy is defined, by
\begin{align*}
&\mathcal{E}_{p}\,:\,  L^{2}(\Omega,\mathbb{M}_{\text{sym}}^{2\times2})\times L^{2}(\Omega,\mathbb{M}_{\text{sym}}^{2\times2})\rightarrow \mathbb{R}\\
&(p,p_0)\longmapsto \mathcal{E}_{p}(p,p_0) = \int_{\Omega}H(p-p_0)\,dx,
\end{align*}
and the hardening energy by
\begin{align*}
&\mathcal{E}_{h}\,:\,  L^{2}(\Omega,\mathbb{M}_{\text{sym}}^{2\times2})\rightarrow \mathbb{R}\\
&p\longmapsto \mathcal{E}_{h}(p) = \dfrac{1}{2}\int_{\Omega}Bp:p\,dx.
\end{align*}
Given $\beta_1>0$, $\beta_2>0$ and $h>0$, the visco-elastic energy is 
\begin{align*}
&\mathcal{E}_{ve}\,:\, H^{1}(\Omega,\mathbb{R}^2)\times H^{1}(\Omega,\mathbb{R}^2)\rightarrow \mathbb{R}\\
&(u,u_0)\longmapsto \mathcal{E}_{ve}(u,u_0) = \dfrac{\beta_1}{2h}\int_{\Omega}(E(u)-E(u_0)):(E(u)-E(u_0))\,dx.
\end{align*}
and the viscoplastic energy is defined by
\begin{align*}
&\mathcal{E}_{vp}\,:\,  L^{2}(\Omega,\mathbb{M}_{\text{sym}}^{2\times2})\times L^{2}(\Omega,\mathbb{M}_{\text{sym}}^{2\times2})\rightarrow \mathbb{R}\\
&(p,p_0)\longmapsto \mathcal{E}_{vp}(p,p_0) = \dfrac{\beta_2}{2h}\int_{\Omega}(p-p_0):(p-p_0)\,dx.
\end{align*}
The Griffith surface energy is approximated by the phase-field surface energy  
\begin{align*}
&\mathcal{E}_{S} \,:\,H^{1}(\Omega,\mathbb{R})\rightarrow \mathbb{R}\\
&v\quad\longmapsto \,\mathcal{E}_{S}(v)= \int_{\Omega} \varepsilon \vert\nabla v\vert^{2} dx+ \int_{\Omega}\frac{\left(1-v\right)^{2}}{4\varepsilon} \,dx.
\end{align*}
It is shown in \cite{Bourdin} that in the elastic anti-plane case, where the displacement reduces to a scalar and $Eu$  reduces to $\nabla u$, 
the Ambrosio-Tortorelli functional
$$
\mathcal{E}_{\varepsilon}(\nabla u,v) =\mathcal{E}_{el}(\nabla u,v)+\mathcal{E}_{S}(v),
$$
$\Gamma$-converges, as $0<\eta\ll\epsilon\rightarrow0$, to the Griffith energy $\mathcal{G}$, where
\begin{eqnarray*}
\mathcal{G}(u):=\dfrac{1}{2}\int_{\Omega}A |\nabla u|^2\,dx+\mathcal{H}^{N-1}(S(u)).
\end{eqnarray*}
Here, $S(u)$ denotes the discontinuity set of u, and $\mathcal{H}^{N-1}$ is the $(N-1)$- dimensional Hausdorff measure.

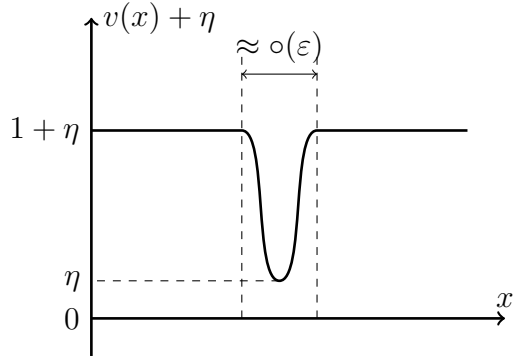
\begin{figure}[h]
\begin{center}	
	\begin{tikzpicture}[scale=1] 
		\def\VDepth{2}; 
		\def\VDefaultLevel{2.5}; 
		\def\VEpsilon{0.5}; 
		\def\VPower{0.35}; 
		\def\VBorder{2}; 
		
		\draw[line width=1pt] ({-\VEpsilon - \VBorder}, \VDefaultLevel) node[left] {$1+\eta$} 
			-- ++(\VBorder, 0)
			.. controls +(0:\VPower) and +(180:\VPower) .. ++(\VEpsilon, -\VDepth)
			.. controls +(0:\VPower) and +(180:\VPower) .. ++(\VEpsilon, \VDepth)
			-- ++(\VBorder, 0);
		
		\draw [dashed] (0, {\VDefaultLevel - \VDepth}) -- ++({-\VEpsilon - \VBorder}, 0) node[left] {$\eta$};		
			
		\draw[dashed] (-\VEpsilon, 0) -- (-\VEpsilon, {\VDefaultLevel + 1});
		
		\draw[dashed] (\VEpsilon, 0) -- (\VEpsilon, {\VDefaultLevel + 1});
		
		
		\draw[<->] (-\VEpsilon, {\VDefaultLevel + 0.75}) -- ++({2*\VEpsilon}, 0) node[pos=0.5, above] {$\approx\circ(\varepsilon)$};
		
		
		\draw[->, line width=1pt] ({-\VEpsilon - \VBorder}, -0.5) -- ++(0, {\VDefaultLevel + 2}) node[right] {$v(x) + \eta$};
		\draw[->, line width=1pt] ({-\VEpsilon - \VBorder}, 0) node[left] {$0$} -- ++({2*\VBorder + 2*\VEpsilon + 0.5}, 0) node[above] {$x$};
		\end{tikzpicture}
\end{center}
\caption{In the generalized crack model, a crack is replaced by a thin region of very compliant material.}
\label{rupture}
\end{figure}
Note that in the region around an approximate crack, where $v$ is close to 0,
the effective Lam\'e tensor is $(v^2+\eta)A$ : the elastic material is replaced there by a 
very compliant medium. We define $a:=v^2+\eta$.
\medskip
\subsection{Formulation of the models.}
We now propose 3 models that combine the various ingredients that we are interested in. 
\begin{itemize}
 \item Model 1 contains: elasticity, plasticity, visco-elasticity and fracture,
 \item Model 2: elasticity, plasticity, visco-plasticity, fracture,
 \item Model 3: elasticity, plasticity, kinematic hardening, fracture.
\end{itemize}
We define a time evolution for our models to be a quadruplet of functions
$(u,e,p,v): \Omega\times[0,T_f]\rightarrow \mathbb{R}^{2} \times\mathbb{M}_{\text{sym}}^{2\times2}\times\mathbb{M}_{\text{sym}}^{2\times2} \times \mathbb{R}$
satisfying the following conditions:
\begin{itemize}
 \item[(E1)] Initial condition: $(u(0),v(0),e(0),p(0))=(u_0,v_0,e_0,p_0)$ with \\ 
 $(u_0,e_0,p_0) \in A_{adm}(w(0))$. We also suppose that $(v_0^2+\eta)|Ae_0|\leq\tau$ and
$v_0 =1$ in $\Omega$ (the medium at $t=0$ does not contain any crack). 
\item[(E2)] Kinematic compatibility: for $t\in[0,T_f]$,
\begin{eqnarray*}
(u(t),e(t),p(t)) \in A_{adm}(w(t)) \label{B} 
\end{eqnarray*}
\item [(E3)] Equilibrium condition: for $t\in[0,T_f]$, 
\begin{equation}\label{bebe}
\left\{
    \begin{array}{ll} 
        -\text{div}(\sigma(t))=f(t),&   a.e.\,\,\text{in}\quad \Omega,\\
         \sigma(t).\vec{n}=g(t), &   \text{on}\quad\partial\Omega_N,\notag\\
         (u(t),v(t))=(w(t),1), \quad &\text{on} \quad \partial\Omega_D. \notag
      \end{array}
\right.
\end{equation}
\item [(E4)] Constitutive relations: for  $t\in[0,T_f]$,
\begin{itemize}
 \item Model 1: $\sigma(t)= (v(t)^2 + \eta)Ae(t) + \beta_1 E\dot{u}(t)$.
The first term represents the stress due to elastic deformation, while the second 
represents viscous dissipation.
\item Model 2 and Model 3: $\sigma(t)= (v(t)^2 + \eta)Ae(t).$
There the stress is only related to elastic deformation. We recall the notation $a(t):=v(t)^2 + \eta$.
\end{itemize}
\item[(E5)] Plastic flow rule: for a.e. $t\in[0,T_f]$,
\begin{itemize}
 \item Model 1: 
 \begin{eqnarray}
a(t)Ae(t) \in \partial H(\dot{p}(t)) \quad \text{for}\,\,a.e.\quad x\in \Omega. \label{model1}
\end{eqnarray}
\item Model 2: 
\begin{eqnarray}
a(t)Ae(t)\in \partial H_{\beta_2}(\dot{p}(t)) \quad \text{for}\,\,a.e.\quad x\in \Omega. \label{model2} 
\end{eqnarray}
\item Model 3:
\begin{eqnarray}
a(t)Ae(t)-Bp(t) \in \partial H(\dot{p}(t)) \quad \text{for}\,\,a.e.\quad x\in \Omega. \label{model3}
\end{eqnarray}
\end{itemize}
\item[(E6)] Crack stability condition: for $t\in[0,T_f]$,
\begin{eqnarray*}
\mathcal{E}_{el}(e(t),v(t))+\mathcal{E}_{S}(v(t))= \inf_{v=1\,\text{sur} \,\partial\Omega_D,v\leq v(t)} \mathcal{E}_{el}(e(t),v)+\mathcal{E}_{S}(v).\label{ruptura}
\end{eqnarray*}
The crack stability condition implies that a fracture can only grow, and cannot disappear.
\item[(E7)] Energy balance formula: for every $T\in[0,T_f]$, the energy dissipates in the medium so as to satisfy
\begin{itemize}
\item Model 1:
\begin{eqnarray}
\lefteqn{\mathcal{E}_{el}(e(T),v(T))+\mathcal{E}_{S}(v(T))-\langle l(T), u(T)\rangle}\notag\\
&=&
\mathcal{E}_{el}(e(0),v(0))+\mathcal{E}_{S}(v(0))-\langle l(0), u(0)\rangle\notag \\
&-&
\beta_1\int_0^T \parallel E\dot{u}(t)\parallel_2^{2} dt - \tau \int_0^T  \int_\Omega |\dot{p}| dx\,dt\notag\\
&-&
\int_0^T\langle\dot{l},u\rangle\,dt+\int_0^T\int_{\partial\Omega_D}\sigma(t)\vec{n}.\dot{w}(t) \,ds\,dt.\label{energy1}
\end{eqnarray}
\item Model 2:
\begin{eqnarray}
\lefteqn{\mathcal{E}_{el}(e(T),v(T))+\mathcal{E}_{S}(v(T))-\langle l(T), u(T)\rangle}\notag\\
&=&
\mathcal{E}_{el}(e(0),v(0))+\mathcal{E}_{S}(v(0))-\langle l(0), u(0)\rangle\notag \\
&-&
\beta_2\int_0^T \parallel\dot{p}(t)\parallel_2^{2} dt - \tau \int_0^T  \int_\Omega |\dot{p}| dx\,dt\notag\\
&-&
\int_0^T\langle\dot{l},u\rangle\,dt+\int_0^T\int_{\partial\Omega_D}\sigma(t)\vec{n}.\dot{w}(t) \,ds\,dt.\label{energy2}
\end{eqnarray}
\item Model 3:
\begin{eqnarray}
\lefteqn{\mathcal{E}_{el}(e(T),v(T))+\mathcal{E}_{S}(v(T))+\mathcal{E}_{h}(p(t))-\langle l(T), u(T)\rangle}\notag\\
&=&
\mathcal{E}_{el}(e(0),v(0))+\mathcal{E}_{S}(v(0))+\mathcal{E}_{h}(p(0))-\langle l(0), u(0)\rangle\notag\\
&-&
\tau \int_0^T \int_\Omega |\dot{p}| dx\,dt\notag-\int_0^T\langle\dot{l},u\rangle\,dt+\int_0^T\int_{\partial\Omega_D}\sigma(t)\vec{n}.\dot{w}(t) \,ds\,dt.\notag\\
\label{energy3}
\end{eqnarray} 
\end{itemize}
\end{itemize}
In the rest of paper, we suppose $l\equiv0$.
\section{Consistency of models with thermodynamics.}
In this section, we show that our models can be set in a thermodynamical framework
which resembles that of the Generalized Standard Materials of Halphen and Nguyen~\cite{HalphenSon},
see also Le Tallec~\cite{LeTallec}.
To this end, we introduce for $t\in [0,T]$, a free energy density $w(E(t),v(t),p(t))$ 
which depends on $E(t):=Eu(t)$, and on $v(t), p(t)$, the latter considered as internal variables (see~~\cite{HalphenSon}).
We also introduce a free energy functional $\mathcal{W}(t)$
\begin{eqnarray*}
\mathcal{W}(E,p,v)(t)&=&\int_{\Omega} w(E,p,v)(t)\,dx,
\end{eqnarray*}
and thermodynamic forces as operators associated with the internal variables
\begin{eqnarray*}
T_p(t)\tilde{p}:=-\frac{\partial w}{\partial p}(E,p,v)(t)\tilde{p}\quad\text{and}\quad T_v(t)\tilde{v}:=-\frac{\partial \mathcal{W}}{\partial v}(E,p,v)(t)\tilde{v}.\label{thermo3}
\end{eqnarray*}
Note that in our case, the thermodynamic force associated with the phase field
is defined as a global operator $H^1(\Omega) \longrightarrow \mathbb{R}$.
We also postulate the existence of a dissipation potential $\phi(t)=\phi(\dot{E}(t),\dot{p}(t),\dot{v}(t))$
which is a convex function of its arguments and minimal at
$(\dot{E}(t),\dot{p}(t),\dot{v}(t))=(0,0,0)$.
According to the principle of conservation of linear momentum, 
we recall that the Cauchy theorem implies the existence of symmetric stress tensor $\sigma(t)$ which satisfies  under the hypothesis
of small deformations the equilibrium condition (E3).
The stress tensor is split into an irreversible and a reversible part by setting
\begin{eqnarray}
\sigma^{rev}(t):=\frac{\partial w}{\partial E}(E(t),p(t),v(t))\quad\text{and}\quad \sigma^{irrev}(t):=\sigma(t)-\sigma^{rev}(t).\label{baba}
\end{eqnarray}
Following the work of Halphen and Nguyen \cite{HalphenSon} and LeTallec \cite{LeTallec}, we make the constitutive hypothesis that the thermodynamic forces are
related to the dissipation potential by 
\begin{eqnarray}
(\sigma^{irrev}(t),T_p(t), T_v(t)) \in \partial \phi (\dot{E}(t), \dot{p}(t), \dot{v}(t)). \label{SE8}
\end{eqnarray}
Our goal is to show that our models are consistent with this thermodynamic 
framework, in the sense that if one assumes  (\ref{SE8}) and if the equilibrium condition (E3) is verified, then one recovers the relations (E4)-(E7),
and further, a form of the Clausius-Duhem inequality holds.
We state this result for Model 1 only. The same analysis carries on
for Models 2 and 3 (see the remark below).
We also define the fracture dissipation potential by
$$
D_{\text{S}} (\xi)= \left\{
    \begin{array}{ll}
       0 & \mbox{if} \quad\xi\leq 0, \,\, a.e.\quad\text{in} \,\Omega, \,\xi\in H^1_D(\Omega),\, \\
\infty & \mbox{elsewhere},
    \end{array}
\right.
$$
with  $H^1_D(\Omega):=\{z\in H^1(\Omega); z=0\,\, \text{on}\,\, \partial\Omega_D\}$.
\begin{theorem}\label{Theorem1}
Suppose that $(u,v,e,p)$ in $C^1(0,T_f, H^1(\Omega))\times C^1(0,T_f, H^1(\Omega))\times C^1(0,T_f, L^2(\Omega))\times C^1(0,T_f, L^2(\Omega))$ 
satisfy for all $t\in [0,T]$, $(\dot{u}(t),\dot{e}(t),\dot{p}(t))\in A_{adm}(\dot{w}(t))$,  $\dot{v}(t)\leq 0 \, a.e.\,\text{in} \,\,\Omega, \,\dot{v}(t)\in H^1_D(\Omega)$, $v(t)=1$ on $\partial\Omega_D$,
 (E1), (E2), (E3), and (\ref{SE8}). Let
\begin{eqnarray*}
\mathcal{W}_1(t,Eu(t),p(t),v(t))&:=&  \dfrac{1}{2} \int_{\Omega}\left( v(t)^{2}+\eta\right)A(Eu(t)-p(t)):(Eu(t)-p(t)) \,dx\\
                  &+& \int_{\Omega} \varepsilon \vert\nabla v(t)\vert^{2} dx+ \int_{\Omega}\frac{\left(1-v(t)\right)^{2}}{4\varepsilon} \,dx,
\end{eqnarray*}
and the potential of dissipation
\begin{eqnarray*}
\phi_1(t,E\dot{u}(t),\dot{p}(t),\dot{v}(t))= \dfrac{1}{2} \beta_1E\dot{u}(t):E\dot{u}(t)+ \tau |\dot{p}(t)|+D_{\text{S}} (\dot{v}(t)).
\end{eqnarray*}
Then $(u,v,e,p)$ satisfies (E4), (E5), (E6), (E7).  Furthermore, for all $t\in [0,T_f]$,
\begin{eqnarray}
\mathcal{D}(t):=\int_{\Omega}\sigma(t):E\dot{u}(t)\,dx-\dot{\mathcal{W}}_1(t)\geqslant 0.
\end{eqnarray}
\end{theorem}
\begin{proof}:
Let $t\in[0,T_f]$. The relations (\ref{baba}) and (\ref{SE8}) lead to (E4)
\begin{eqnarray}
\sigma(t)=a(t)Ae(t)+\beta_1E\dot{u}(t). \label{SE10}
\end{eqnarray}
We also deduce from (\ref{SE8}) that 
\begin{eqnarray}
(v(t)^2+\eta)Ae(t) \in \partial H(\dot{p}(t)). \label{SE11}
\end{eqnarray}
which prove (E5). 
From (\ref{SE8}) we see that for every $\xi\leq0$ and $\xi=0$ on $\partial\Omega_D$
\begin{eqnarray}
\lefteqn{-\frac{\partial \mathcal{W}_1}{\partial v}(E,p,v)(t)(\xi-\dot{v}(t))}\notag\\
&=&
-\int_{\Omega}v(t)A(Eu(t)-p(t)):(Eu(t)-p(t))(\xi-\dot{v}(t))\notag\\
&+&
(2\varepsilon)^{-1}(v(t)-1)(\xi-\dot{v}(t))+2\varepsilon\nabla v(t)\nabla (\xi-\dot{v}(t))\,dx\leq0,\label{SE12}
\end{eqnarray}
so that we obtain
\begin{eqnarray}
\lefteqn{\int_{\Omega}v(t)A(Eu(t)-p(t)):(Eu(t)-p(t))(\dot{v}(t)-\xi)}\notag\\
&+&
(2\varepsilon)^{-1}(v(t)-1)(\dot{v}(t)-\xi)\,dx+\int_{\Omega}2\varepsilon\nabla v(t)\nabla(\dot{v}(t)-\xi)\,dx\leq0\label{thermottt}.
\end{eqnarray}
Testing (\ref{thermottt}) with $\xi=\dot{v}(t)+\varphi-v(t)$ where $\varphi \in H^1(\Omega)$, $\varphi \leq v(t)$, and $\varphi=1$ on $\partial\Omega_D$, implies that
\begin{eqnarray}
&&2\varepsilon\int_{\Omega} \nabla v(t)\nabla(v(t)-\varphi) \,dx+\int_{\Omega}v(t)A e(t):e(t)(v(t)-\varphi)\,dx\notag\\
&+&(2\varepsilon)^{-1}\int_{\Omega}(v(t)-1)(v(t)-\varphi)\,dx \leq 0,  \label{thermot4}
\end{eqnarray}
for every $\varphi \in H^1(\Omega)$, $\varphi \leq v(t)$,  and  $\varphi =1$ on $\partial\Omega_D$.
We rewrite (\ref{thermot4}) as follows
\begin{eqnarray}
\lefteqn{2\varepsilon\int_{\Omega} \nabla v(t)\nabla v(t) \,dx+\int_{\Omega}v(t)A e(t):e(t)v(t)\,dx+(2\varepsilon)^{-1}\int_{\Omega}(v(t)-1)v(t)\,dx } \notag\\
&\leq&
2\varepsilon\int_{\Omega} \nabla v(t)\nabla \varphi \,dx+\int_{\Omega}v(t)A e(t):e(t)\varphi\,dx+(2\varepsilon)^{-1}\int_{\Omega}(v(t)-1)\varphi\,dx.\notag\\\label{thermot5}
\end{eqnarray}
Using the Cauchy inequality yields
\begin{eqnarray}
 2\varepsilon\int_{\Omega} \nabla v(t)\nabla \varphi \,dx\leq\varepsilon\int_{\Omega} \arrowvert\nabla v(t)\arrowvert^2\,dx+\varepsilon\int_{\Omega} \arrowvert\nabla \varphi\arrowvert^2\,dx,\notag
\end{eqnarray}
\begin{eqnarray}
 \int_{\Omega}v(t)A e(t):e(t)\varphi\,dx\leq\dfrac{1}{2}\int_{\Omega}v^2(t)A e(t):e(t)\,dx+\dfrac{1}{2}\int_{\Omega}\varphi^2 A e(t):e(t)\,dx.\notag
\end{eqnarray}
We rewrite 
\begin{eqnarray*}
&&(v(t)-1)\varphi=(v(t)-1)(\varphi-1)+(v(t)-1),\\ 
&&(v(t)-1)v(t)-(v(t)-1)=(v(t)-1)^2,
\end{eqnarray*}
and it follows that
\begin{eqnarray}
\mathcal{E}_{el}(e(t),v(t))+ \mathcal{E}_{S}(v(t))\leq \mathcal{E}_{el}(e(t),\varphi)+\mathcal{E}_{S}(\varphi)
\end{eqnarray}
for all $\varphi \in H^1(\Omega)$, $\varphi \leq v(t)$,  and  $\varphi =1$ on $\partial\Omega_D$, which proves (E6).
We now prove the energy balance formula. First we differentiate $\mathcal{W}_1(t,E(u(t)),p(t),v(t))$ in time:
\begin{eqnarray}
\lefteqn{\dfrac{d}{dt} \mathcal{W}_1(t,E(u(t)),p(t),v(t))=\int_{\Omega}a(t)A(Eu(t)-p(t)):(E\dot{u}(t)-\dot{p}(t))\,dx}\notag\\
&+&\int_{\Omega}v(t)A(Eu(t)-p(t)):(Eu(t)-p(t))\dot{v}(t)+(2\varepsilon)^{-1}(v(t)-1)\dot{v}(t)\,dx\notag\\
&+&2\int_{\Omega}\varepsilon\nabla v(t)\nabla \dot{v}(t)\,dx\label{I1}.
\end{eqnarray}
Testing inequality (\ref{SE12}) with $\xi=0$ and $\xi=2\dot{v}(t)$ leads to 
\begin{eqnarray}
-\frac{\partial \mathcal{W}_1}{\partial v}(E,p,v)(t)\dot{v}(t)=0.\label{I2}
\end{eqnarray}
From (\ref{I1}) and (\ref{I2}) we deduce that
\begin{eqnarray}
\lefteqn{\dfrac{d}{dt} \mathcal{W}_1(t,Eu(t),p(t),v(t))}\notag\\
&=&
\int_{\Omega}a(t)A(Eu(t)-p(t)):(E\dot{u}(t)-\dot{p}(t))\,dx \notag\\
&=&\int_{\Omega}\sigma(t):E\dot{u}(t)\,dx-\int_{\Omega}\beta_1 E\dot{u}(t):E\dot{u}(t)\,dx\notag\\
&-&\int_{\Omega} a(t)A(Eu(t)-p(t)):\dot{p}(t)\,dx.\label{I3}
\end{eqnarray}
The equilibrium equation (E3) gives 
\begin{eqnarray}
\int_{\Omega}\sigma(t):E\dot{u}(t)\,dx=\int_{\partial\Omega_D}\sigma(t)\vec{n}.\dot{w}(t) \,ds.\label{I4}
\end{eqnarray}
By definition of the subgradient, (\ref{SE11}) leads to a variational inequality: for all admissible $q \in L^2(\Omega,\mathbb{M}_{\text{sym}}^{2\times2})$ we have  
\begin{eqnarray}
\tau\int_{\Omega}|q|\,dx \ge \tau\int_{\Omega}|\dot{p}(t)|\,dx+\int_{\Omega} a(t)A(Eu(t)-p(t)) :(q-\dot{p}(t))\,dx \label{model4}
\end{eqnarray}
Testing (\ref{model4}) with $q=0$ et $q=2\dot{p}(t)$ implies that
\begin{eqnarray}
\int_{\Omega} a(t)A(Eu(t)-p(t)) :\dot{p}(t)\,dx=\tau\int_{\Omega}|\dot{p}(t)|dx.\label{mod2}
\end{eqnarray}
So that we deduce from (\ref{I3}), (\ref{I4}), (\ref{mod2}) that
\begin{eqnarray}
\dfrac{d}{dt} \mathcal{W}_1(t,Eu(t),p(t),v(t))&=& -\int_{\Omega}\beta_1 E\dot{u}(t):E\dot{u}(t)\,dx-\tau\int_{\Omega}|\dot{p}(t)|\,dx\notag\\
 &+&\int_{\partial\Omega_D}\sigma(t)\vec{n}.\dot{w}(t) \,ds. \label{mod3}
\end{eqnarray}
Integrating (\ref{mod3}) over $[0, T]$, for every $0\leq T\leq T_f$ shows that the balance formula (E7) holds.
Finally, from  (\ref{I4}) and (\ref{mod3}) we deduce that
\begin{eqnarray*}
\lefteqn{\mathcal{D}(t):=\int_{\Omega}\sigma(t):E\dot{u}(t)\,dx-\dot{\mathcal{W}_1}(t)}\\
&=&
\int_{\Omega}\beta_1 E\dot{u}(t):E\dot{u}(t)\,dx+\tau\int_{\Omega}|\dot{p}(t)|\,dx\geqslant 0.
\end{eqnarray*}
\hfill$\square$
\end{proof}

\begin{remark}
\begin{enumerate}
 \item The assumption (\ref{SE8}) is stronger than (E6).
 \item Theorem~\ref{Theorem1} also holds for Models 2 and 3 with the following choices of free energies and dissipation potentials:
\begin{itemize}
 \item for Model 2:
\begin{eqnarray*}
\lefteqn{\mathcal{W}_2(t,Eu(t),p(t),v(t))}\\
&:=&
\dfrac{1}{2} \int_{\Omega}\left( v(t)^{2}+\eta\right)A(Eu(t)-p(t)):(Eu(t)-p(t)) \,dx\\
\quad\quad&+& 
\int_{\Omega} \varepsilon \vert\nabla v(t)\vert^{2} dx+ \int_{\Omega}\frac{\left(1-v(t)\right)^{2}}{4\varepsilon} \,dx.
\end{eqnarray*}
and
\begin{eqnarray*}
\phi_2(t,E\dot{u}(t),\dot{p}(t),\dot{v}(t))= \dfrac{1}{2} \beta_2\dot{p}(t):\dot{p}(t)+ \tau |\dot{p}(t)|+D_{S}(\dot{v}(t)).
\end{eqnarray*}
 \item for Model 3:
\begin{eqnarray*}
\lefteqn{\mathcal{W}_3(t,Eu(t),p(t),v(t))}\\
&:=&  \dfrac{1}{2} \int_{\Omega}\left( v(t)^{2}+\eta\right)A(Eu(t)-p(t)):(Eu(t)-p(t)) \,dx\notag\\
&+&
\dfrac{1}{2} \int_{\Omega} Bp(t):p(t)\,dx+\int_{\Omega} \varepsilon \vert\nabla v(t)\vert^{2} dx+ \int_{\Omega}\frac{\left(1-v(t)\right)^{2}}{4\varepsilon} \,dx,
\end{eqnarray*}
and
\begin{eqnarray*}
\phi_3(t,E\dot{u}(t),\dot{p}(t),\dot{v}(t))= \tau |\dot{p}(t)|+D_{S} (\dot{v}(t)).
\end{eqnarray*}
\end{itemize}
\end{enumerate}
\end{remark}
\section{Discrete-time evolutions for Models 1-3.}
We now approximate the continuous-time  evolutions of the constructed models via discrete time evolutions obtained by solving incremental variational problems. 
We describe the discrete-time evolution of the medium as follows:
we consider a partition of the time interval $[0,T_f]$ into $N_f$ sub-intervals of equal length $h$:
\begin{eqnarray*}
 0=t_h^0<t_h^1<...<t_h^n<...<t_h^{N_f}=T_f, \quad \text{with}\quad h=\dfrac{T_f}{N_f}= t_h^n-t_h^{n-1}\rightarrow0.
\end{eqnarray*}
We define
\begin{eqnarray*}
\mathcal{B}(t_h^n)&:=&\left\{(z,q,\varphi)\in H^1(\Omega)\times L^2(\Omega)\times H^1(\Omega);
\begin{array}{ll} 
        z=w_h^n ,&\text{on}\quad \partial \Omega_D,\\
         \varphi=1, & \text{on}\quad\partial \Omega_D,\notag\\
         \varphi\leq v_h^{n-1},&\text{in}\quad\Omega.\notag
      \end{array}
\right\}.
\end{eqnarray*}
Let us assume that for $n\geqslant1$, the approximate evolution $(u_h^{n-1}, v_h^{n-1}, p_h^{n-1})\in\mathcal{B}(t_h^{n-1})$ is known at $t_h^{n-1}$.
We seek $(u_h^{n}, v_h^{n}, p_h^{n})$ at time $t_h^n$ as the solution to the following variational problem:  
\begin{eqnarray}
 \min_{(z,q,\varphi)\in \mathcal{B}(t_h^n)}\mathcal{E}_{total}(z,q,\varphi,u_h^{n-1}, v_h^{n-1}, p_h^{n-1}),\label{Problem1}
\end{eqnarray}
where, for each model, $\mathcal{E}_{total}$ is defined as follows:
\begin{enumerate}
 \item Model 1: Elasto-plasticity, viscoelasticity and fracture:
\begin{eqnarray*}
\mathcal{E}^1_{total}(z,q,\varphi,u_h^{n-1}, v_h^{n-1}, p_h^{n-1})&=&\mathcal{E}_{el}(\varphi,E(z)-q)+\mathcal{E}_{p}(q,p_h^{n-1})\\
&+&\mathcal{E}_{ve}(z,u_h^{n-1})+\mathcal{E}_{S}(\varphi).
\end{eqnarray*}
 \item Model 2: Elasto-plasticity, viscoplasticity and fracture:
\begin{eqnarray*}
\mathcal{E}^2_{total}(z,q,\varphi,u_h^{n-1}, v_h^{n-1}, p_h^{n-1})&=&\mathcal{E}_{el}(\varphi,E(z)-q)+\mathcal{E}_{p}(q,p_h^{n-1})\\
&+&\mathcal{E}_{vp}(q,p_h^{n-1})+\mathcal{E}_{S}(\varphi).
\end{eqnarray*}
 \item Model 3: Elasto-plasticity, linear kinematic hardening and fracture:
\begin{eqnarray*}
\mathcal{E}^3_{total}(z,q,\varphi,u_h^{n-1}, v_h^{n-1}, p_h^{n-1})&=&\mathcal{E}_{el}(\varphi,E(z)-q)+\mathcal{E}_{p}(q,p_h^{n-1})\\
&+&\mathcal{E}_{h}(q)+\mathcal{E}_{S}(\varphi).
\end{eqnarray*}
One can easily prove that for $i=1,2,3$ the variational problem
\begin{eqnarray}
 \min_{(z,q,\varphi)\in \mathcal{B}(t_h^n)}\mathcal{E}^i_{total}(z,q,\varphi,u_h^{n-1}, v_h^{n-1}, p_h^{n-1}).\label{Problem2}
\end{eqnarray}
has at least one solution. If $(z_n,q_n,\varphi_n)_n$ is a minimizing subsequence, that one easily checks that $\parallel z_n \parallel_{H^1}$, $\parallel p_n \parallel_{L^2}$, 
$\parallel  \varphi_n \parallel_{H^1}$ are uniformly bounded, so that a subsequence converges weakly to some $(z,q,\varphi)$.
The only difficulty in passing to the limit in $\mathcal{E}_{total}$ comes from the term
\begin{eqnarray*}
\int_{\Omega} \varphi_n^2 A(Ez_n-q_n):(Ez_n-q_n)\,dx,
\end{eqnarray*}
which can be rewritten as
\begin{eqnarray*}
\int_{\Omega}  A\varphi_n(Ez_n-q_n):\varphi_n(Ez_n-q_n)\,dx,
\end{eqnarray*}
and one can use the fact that since
\begin{eqnarray*}
& \displaystyle Ez_n \rightharpoonup Ez&  \quad \text{weakly in} \quad L^{2}, \\
& \displaystyle \varphi_n \rightharpoonup \varphi&  \quad \text{weakly in} \quad H^{1}, \\
& \displaystyle q_n \rightharpoonup q&  \quad \text{weakly in} \quad L^{2},
\end{eqnarray*}
one has
\begin{eqnarray*}
\displaystyle \varphi_n(Ez_n-q_n) \rightharpoonup \varphi(Ez-q)&  \quad \text{weakly in} \quad L^{2}.
\end{eqnarray*}
Note however that $\mathcal{E}_{total}$ is not convex because of the no quadratic term $\varphi^2 (E(z)-q):(E(z)-q)$ in the elastic energy.
So that there might be several solutions to (\ref{Problem2}).
\end{enumerate}
\subsection{An alternate minimization algorithm and backtracking for materials with memory}
A solution $(u_h^n, v_h^n, p_h^n)$ of problem (\ref{Problem2}) is characterized by a system of one equality and two variational inequalities. Such a system is not easy to solve numerically.
For this reason we propose to solve (\ref{Problem2}) at each time step $t_h^n$ using an alternate minimization algorithm. The advantage of this approach
is that the problem (\ref{Problem2}) is separately strictly convex in each variable, so that each alternating step has a unique solution.
\subsubsection{An alternate minimization algorithm}
 \begin{algorithm}
\caption{Alternate minimization algorithm}
Let $\delta_1>0$ and $\delta_2>0$ be fixed tolerance parameters.
\begin{enumerate}
\item Let $m=0$, $v^n_{(m=0)}=v_h^{n-1}$ 
\item \textbf{iterate}
\item Find $(u^n_{(m)}, p^n_{(m)}):={\argmin}_{(u,p)} \mathcal{E}_{total}(u,p,v^n_{(m-1)})$
\item \hspace*{2cm} Let $l=0$, $p^n_{(l=0)}=p_h^{n-1}$
\item \hspace*{2cm} \textbf{iterate} 
\item \hspace*{2cm} $u^n_{(l)}:={\argmin}_{u} \mathcal{E}_{total}(u,p^n_{(l-1)},v^n_{(m-1)} )$
\item \hspace*{2cm} $p^n_{(l)}:={\argmin}_{p} \mathcal{E}_{total}(u^n_{(l)},p,v^n_{(m-1)} )$
\item \hspace*{2cm} \textbf{until} $\parallel u^n_{(l)}-u^n_{(l-1)}\parallel_{H^1}\leq\delta_1$
\item \hspace*{2cm} We define $u^n_{(m)}:=u^n_{(l)}$ and $p^n_{(m)}:=p^n_{(l)}$ at convergence
\item Find $v^n_{(m)}:={\argmin}_{v} E_{total}(u^n_{(m)},p^n_{(m)},v)$
\item \textbf{until} $\parallel v^n_{(m)}-v^n_{(m-1)}\parallel_{H^1}\leq\delta_2$
\item We define $u_h^n:=u^n_{(m)}$, $p_h^n:=p^n_{(m)}$ and $v_h^n:=v^n_{(m)}$ at convergence
\end{enumerate}
\end{algorithm}
In practice, it is not exactly the variational problems described
in the alternating procedure above, that one solves, but the associated
first-order optimality conditions of problems that have been discretized 
in space too. This may introduce local minima, as the following example of a traction 
experiment of a 1D bar illustrates.
Assume that $u, v \in W^{1,\infty}(0,T_f,\Omega)$ represent the displacement
and phase-field marker of a 1D bar that lies in $\Omega = (0,L)$. We consider
a simple model of evolution with only elasticity and fracture where the total energy writes
\begin{eqnarray*}
\mathcal{E}_{total}(u,v)&=&\mathcal{E}_{el}(v,u^{'}) + \mathcal{E}_S(v)\\
&=&\dfrac{1}{2} \int_{\Omega} (v^2 + \eta) K (u^{'})^2 \,dx+\int_{\Omega }\varepsilon(v^\prime)^2 +\frac{(1-v)^2}{4 \varepsilon}\,dx,
\end{eqnarray*}
where $K > 0$ is a fixed Young modulus and where the primes denote derivatives
with respect to $x$. The bar is crack-free at $t=0$ and thus $v(x,0) = 1$.
It is fixed at $x=0$, while a uniform traction $u(L,t) = tL$ is applied at the other extremity.
If $u^\prime$ is close to a constant at time $t$, say $u^\prime(x,t) \sim t$, then the Euler-Lagrange optimality
condition for minimization of the total energy with respect to $v$ amounts to solving
\begin{eqnarray*}
v^{\prime\prime} -\left(\frac{1}{4\varepsilon^2}+\dfrac{Kt^2}{2\varepsilon}\right) v+\frac{1}{4\varepsilon^2}=0,
\end{eqnarray*}
the solution of which is
\begin{eqnarray*}
v(x,t)=C_1e^{-\frac{x\sqrt{2}\sqrt{4Kt^2\varepsilon+2}}{4\varepsilon}}+C_2e^{\frac{x\sqrt{2}\sqrt{4Kt^2\varepsilon+2}}{4\varepsilon}},
\end{eqnarray*}
with
\begin{eqnarray*}
C_1:=\frac{e^{\frac{L\sqrt{2}\sqrt{4Kt^2\varepsilon+2}}{4\varepsilon}}-1}{\left(2kt^2\varepsilon+1\right) \left( e^{-\frac{L\sqrt{2}\sqrt{4Kt^2\varepsilon+2}}{4\varepsilon}}-e^{\frac{L\sqrt{2}\sqrt{4Kt^2\varepsilon+2}}{4\varepsilon}}\right)},
\end{eqnarray*}
and
\begin{eqnarray*}
C_2:=\frac{e^{-\frac{L\sqrt{2}\sqrt{4Kt^2\varepsilon+2}}{4\varepsilon}}-1}{\left(2kt^2\varepsilon+1\right) \left( e^{-\frac{L\sqrt{2}\sqrt{4Kt^2\varepsilon+2}}{4\varepsilon}}-e^{\frac{L\sqrt{2}\sqrt{4Kt^2\varepsilon+2}}{4\varepsilon}}\right)}.
\end{eqnarray*}
These profiles are indeed what one obtains in the course of the numerical
computations according to the algorithm described above, 
until $t$ reaches a sufficiently large value so that the
term $\int_\Omega (v^2 + \eta) K (u^{'})^2\,dx$ dominates $\int_\Omega \frac{(1-v)^2}{4\varepsilon}\,dx$ 
in the energy, see Figures~\ref{PV}, ~\ref{DERIVEEU} and \ref{ener} (we use the same parameters as in \cite{BourdinFrancfortMarigo} by Bourdin).
Note that, due to the presence of the exponentials in the expression of $v$,
these profiles vary significantly near the extremities $x=0$ and $x=L$ of the beam, 
but are quite flat otherwise, and do not correspond to the picture of a generalized
crack as that depicted in Figure~\ref{rupture}.
Further, the corresponding states $u,v$ are only local minima,
as one can build states with lower total energy, as the examples below show.
We note that choosing $\varepsilon$ smaller does not improve the situation for that matter.
This hurdle had been noticed earlier by Bourdin~\cite{Bourdin2}, \cite{Bourdin}, who suggested to complement
the numerical algorithm with a supplementary step called backtracking, 
where after each iteration in time, one imposes a necessary condition derived 
from the definition~(\ref{Problem1}) of the discrete-time evolution. We extend this idea in the context 
of our models, where plasticity and viscous dissipation may occur as well.
\begin{figure}[htbp!]
   \begin{minipage}[c]{.46\linewidth}
\hspace*{-8mm}     
 \includegraphics[width=8cm]{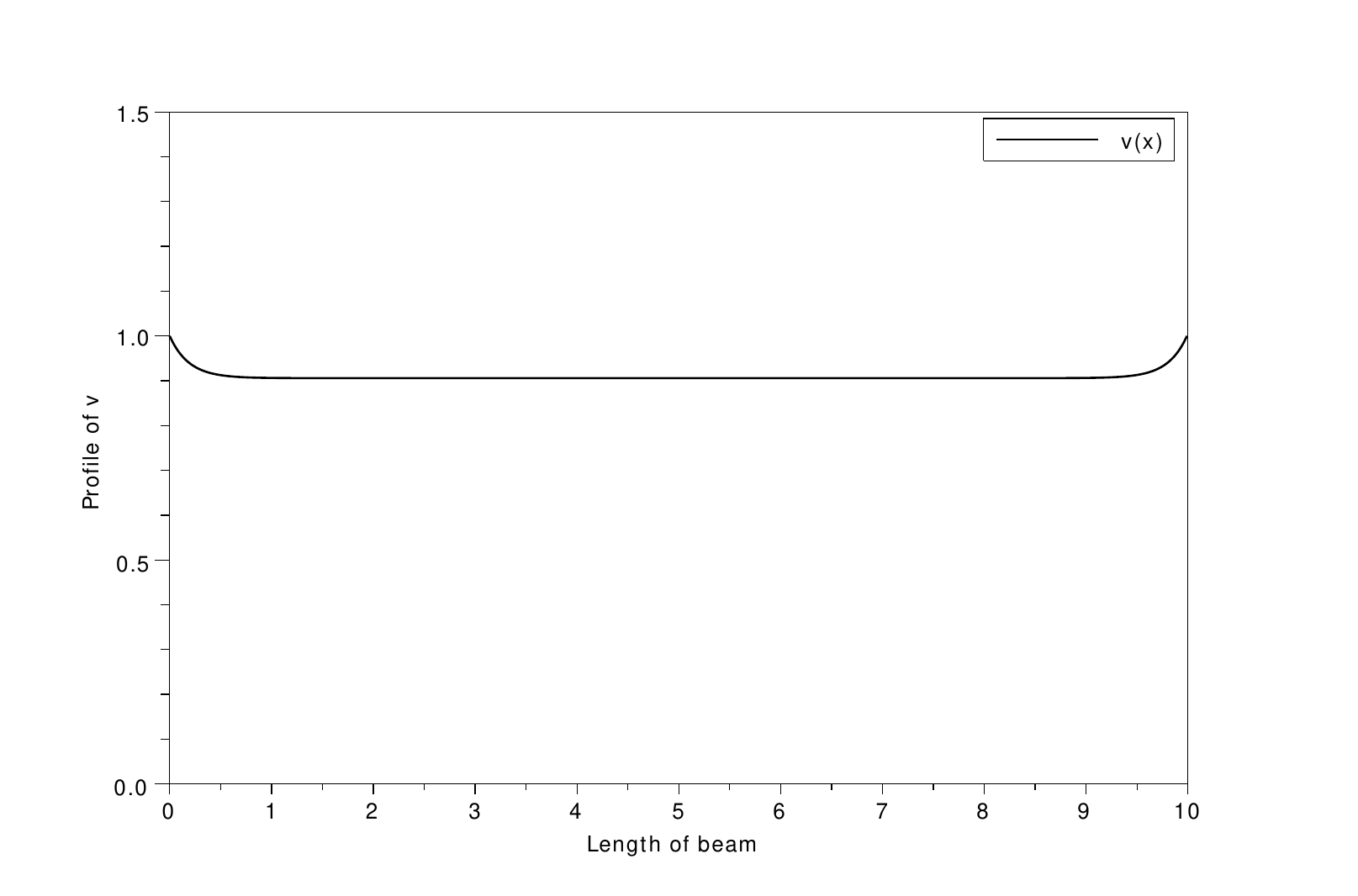}
\caption{Profile of $v(t,.)$ during an elastic evolution with fracture at time $t=4$.}
\label{PV}
 \end{minipage} 
 \hspace*{10mm}
\begin{minipage}[c]{.46\linewidth}
 \includegraphics[width=8cm]{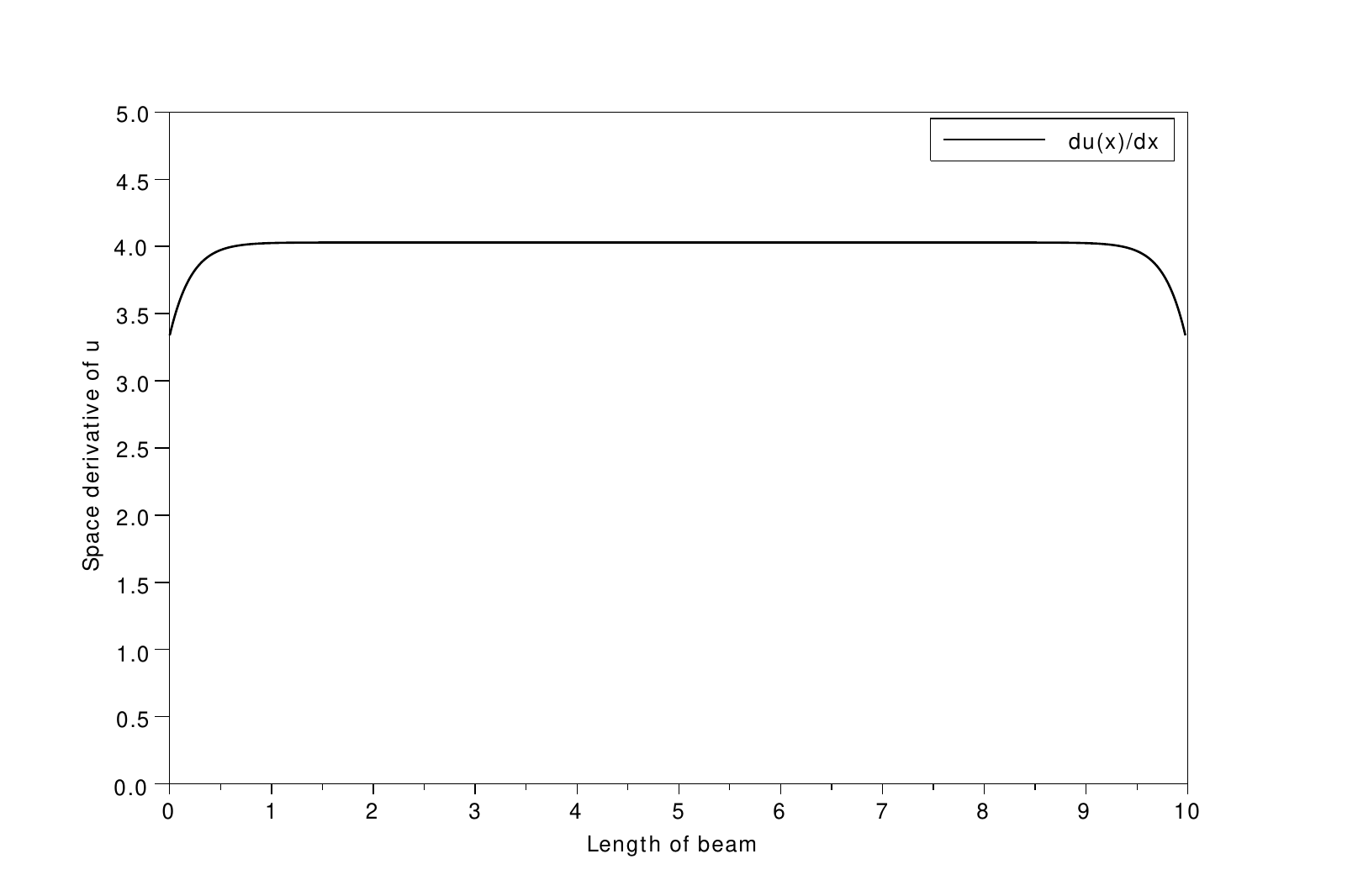}
\caption{Profile of $u{'}(t,.)$  during an elastic evolution with fracture at time $t=4$.}
\label{DERIVEEU}
\end{minipage}
\end{figure}

\begin{figure}[htbp!]
\centering
\hspace*{-8mm}     
 \includegraphics[width=8cm]{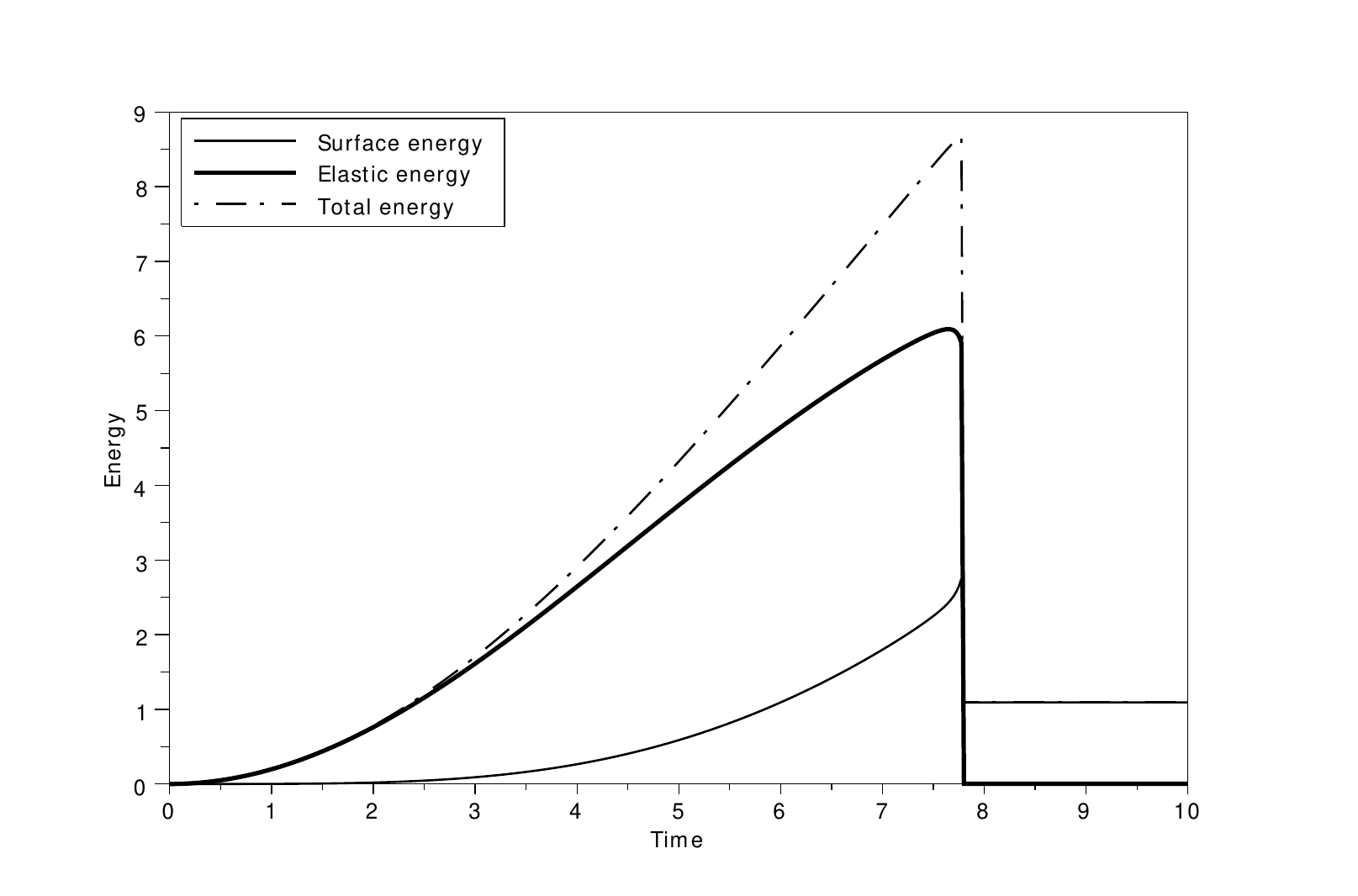}
\caption{Evolution of the total, elastic, and surface energies for the 1D traction experiment without backtracking.}
\label{ener}
\end{figure}
\subsubsection{Backtracking}
Because the loading is monotonous, if $(u_h^n,p_h^n,v_h^n)$ is a solution of (\ref{Problem1}) at time $t_h^n$, 
then $(\frac{t_h^j}{t_h^n} u^n,\frac{t_h^j}{t_h^n} p^n,v_h^n)$ is admissible at time $t_h^j$. Thus, we must have 
\begin{eqnarray}
\mathcal{E}_{total}(u_h^j,p_h^j,v_h^j)\leq \mathcal{E}_{total}((\frac{t_h^j}{t_h^n} u^n,\frac{t_h^j}{t_h^n} p^n,v_h^n)).
\end{eqnarray}
Numerically we check this condition for all $t_h^j\leq t_h^n$. If there exists some j such that this condition is not verified
 $(u_h^j,p_h^j,v_h^j)$ cannot be a global minimizer for time $t_h^j$; and we backtrack to time $t_h^j$ for the alternate minimization algorithm with initialization $v^j_{(m=0)}=v_h^n$ and $p^j_{(l=0)}=p_h^{j-1}$.
\section{Dissipation phenomena appearing during deformation - 1D and 2D numerical experiments}
In this section, we study some evolution problems for Models 1-3 in terms of their  mechanical parameters, to check if
during evolution several dissipation phenomena can be observed.
\subsection{1D-traction numerical experiments with fracture}
We consider a beam $\Omega=(0,L)$ of length L, the Young modulus $K>0$. It is clamped at $x=0$.
A Dirichlet boundary condition $u(L,t)$=tL is imposed at its right extremity $x=L$.
At each time step $t_h^n$ we use  P1-elements to approximate $u$ and $v$, and P0-elements for $p$.
\subsubsection{Elasto-perfectly plastic case with fracture}\label{Model0}
If we take $\beta_1=0$ in Model 1 or $\beta_2=0$ in Model 2, these models reduce to perfect plasticity with numerical fracture.
\begin{eqnarray}
\min_{(z,q,\varphi)\in \mathcal{B}(t_h^n)} \mathcal{E}_{el}(\varphi,E(z)-q)+\mathcal{E}_{p}(q,p_h^{n-1})+\mathcal{E}_{S}(\varphi).\label{backtrack}
\end{eqnarray}
In this example, we illustrate the importance of the backtracking step.
We apply the alternate minimization algorithm without backtracking with the following parameters: $L=10$, $K=4$, $\tau=1.5$, the space discretization mesh size $\bigtriangleup x=0.015$, the time step $h=0.025$, $\eta=10^{-6}$, $\varepsilon=0.094$. 
With this choice of parameters, we observe in Figure~\ref{Figure1} that if the beam is elastic ($v=1$, $p=0$) at time $t=0$, it remains elastic until the time  $t\simeq0.5$ when the beam becomes plastic, then a crack appears at $t\simeq3$.
Because the loading is monotonous, if the system is crack free at $t=0$ and if $p=0$, 
it should remain in the elastic regime until the stress reaches the yield surface or until a crack appears. It is easy to check that if there is no crack,
the yield stress should be reached at time $t_p=\tau/K$, while if there is no plastic deformation, a crack should appear at time $t_c=\sqrt{2/KL}$.
With the given choice of parameters, we obtain $t_p=0.375$ and $t_c=0.224$.
When we compare with Figure~\ref{Figure1}, we see that the beam deforms elastically until plastic deformation takes place at time $t\simeq0.5$,
far from the predicted value.
Figure~\ref{Figure2} shows the same traction experiment computed with the backtracking step. We see that the elastic medium cracks at the computed time $t\simeq0.25$ 
which is close to the expected theoretical
crack time $t_c=0.224$. Before plastic deformation takes place.
\begin{figure}[htbp!]
\begin{center}
 \centering
\includegraphics[width=8cm]{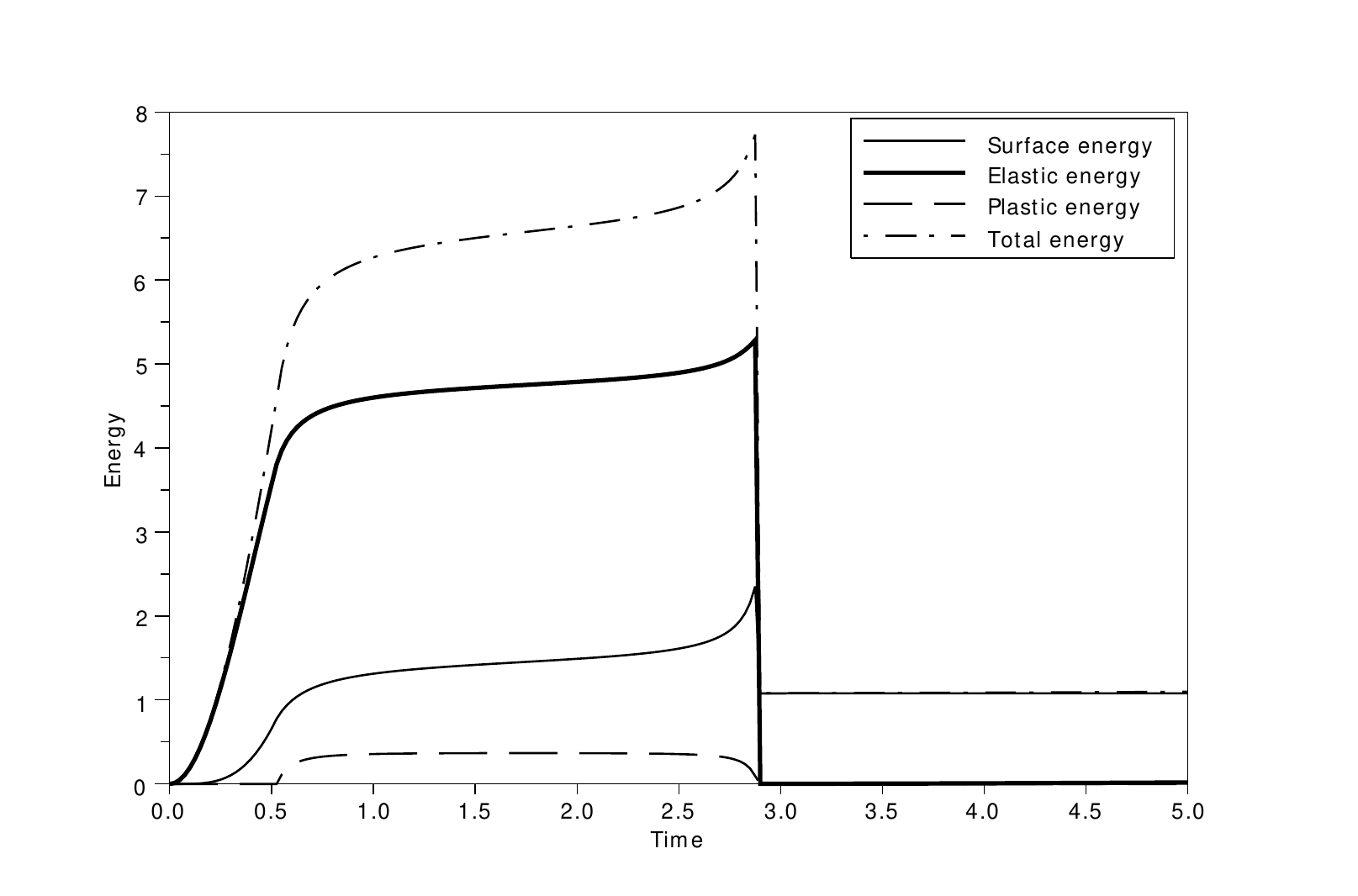}
\end{center}
\caption{Evolution of the total, elastic, plastic and surface energies for the 1D traction experiment without backtracking, $\tau=1.5$.}
\label{Figure1}
\end{figure}
We now change the plastic parameter to $\tau=0.8$ so that the expected plastic time should be $t_p=0.2$. 
Figure~\ref{Figure3} shows that, as expected since now $t_p<t_c$, plastic deformation occurs first. As this model does not allow the elastic energy
to grow once plastic deformation has taken place, no crack appears after $t_p$.
In all the following experiments, the backtracking strategy is used.
\begin{figure}[htbp!]
   \begin{minipage}[c]{.46\linewidth}
\hspace*{-10mm}     
 \includegraphics[width=8cm]{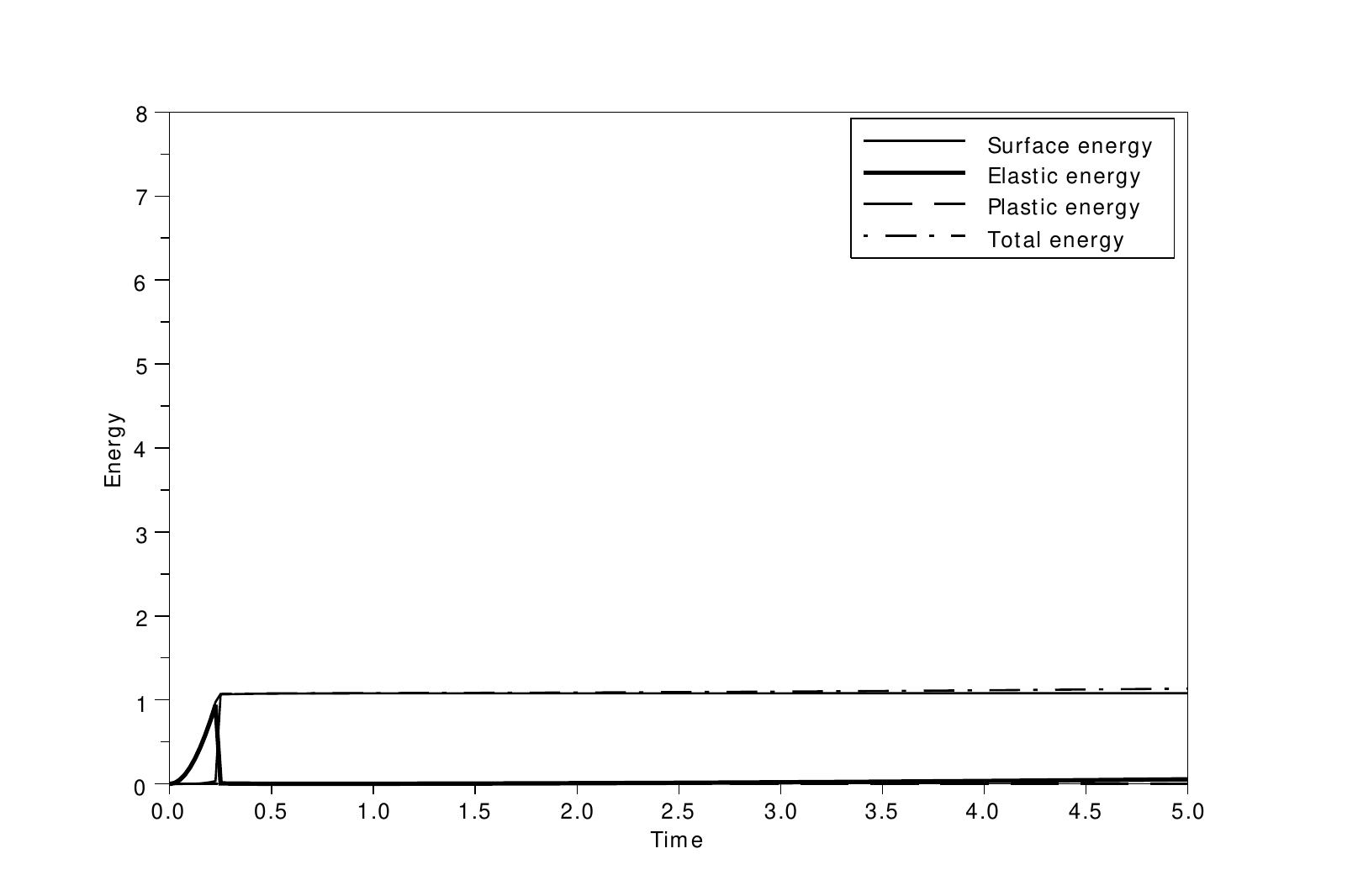}
\caption{Evolution of the total, elastic, plastic and surface energies for the 1D traction experiment with backtracking, $\tau=1.5$.}
\label{Figure2}
   \end{minipage} \hfill
   \begin{minipage}[c]{.46\linewidth}
\hspace*{-10mm}
      \includegraphics[width=8cm]{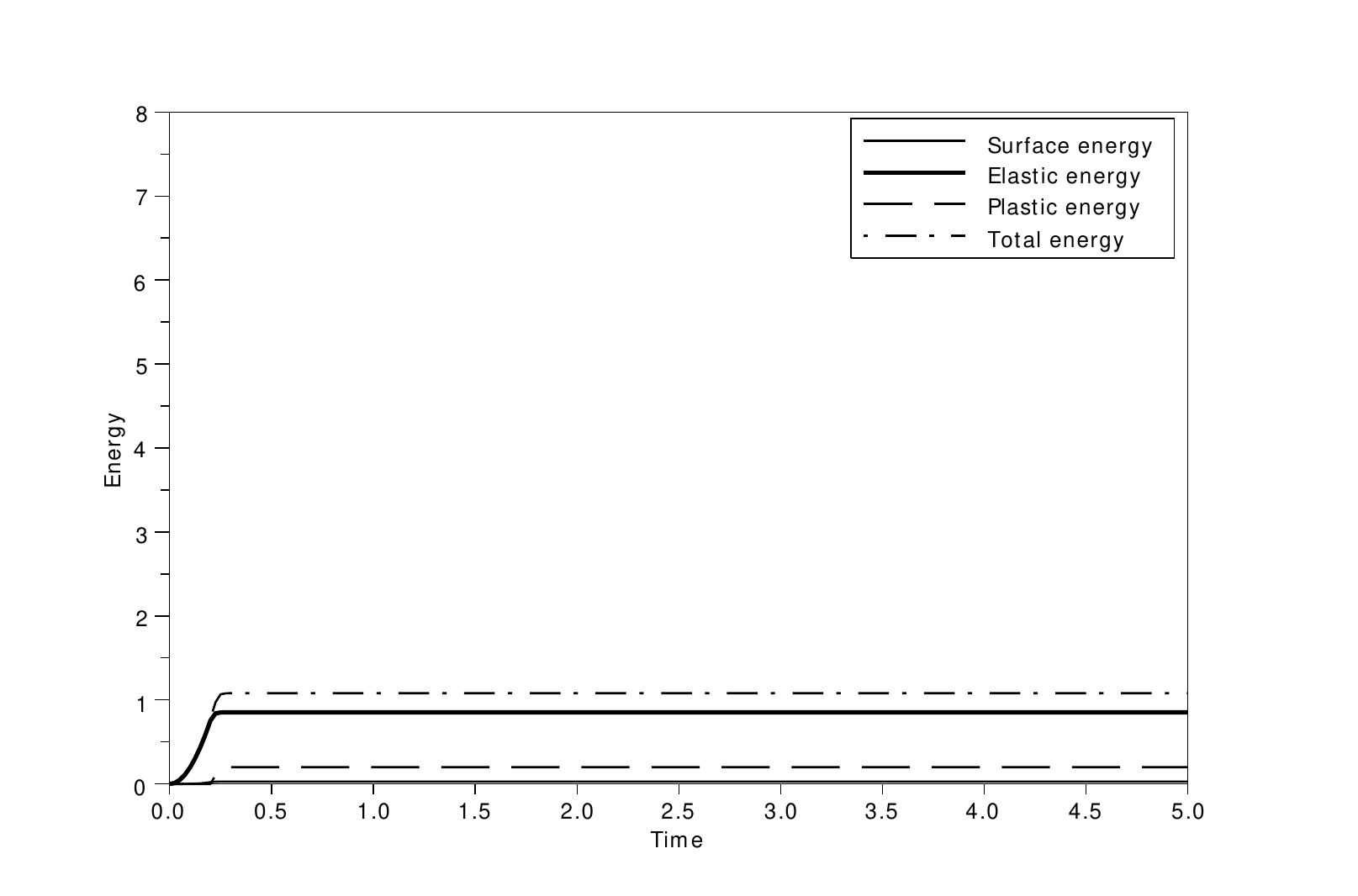}
\caption{Evolution of the total, elastic, plastic and surface energies for the 1D traction experiment with backtracking, $\tau=0.8$.}
\label{Figure3}
   \end{minipage}
\end{figure}
In the sequel, the same discretization parameters are used, and the Young Modulus is chosen as in \ref{Model0}.   
\subsubsection{Model 1 - Elasto-plastic model with visco-elasticity and fracture.}
As can be seen, Model 1 can express the various dissipation mechanisms: elasto-plasticity only (Figure.~\ref{first}), elasticity with fracture (Figure.~\ref{second}),
viscoelasticity with fracture (Figure.~\ref{third}), and elasto-visco-plasticity with fracture (Figure.~\ref{four}).
\begin{figure}[htbp!]
 \begin{minipage}[c]{.46\linewidth}
\hspace*{-10mm}
      \includegraphics[width=8cm]{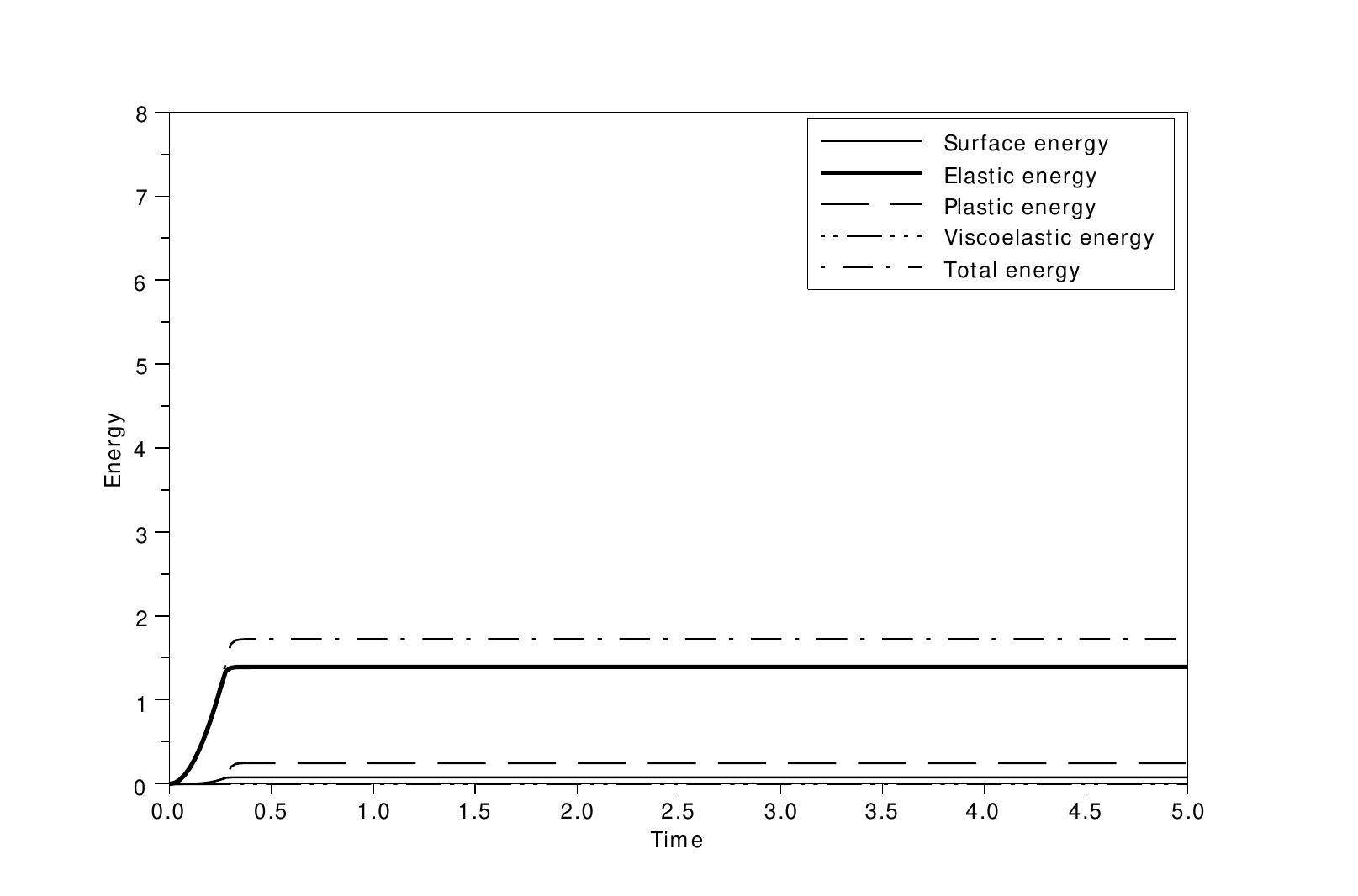}
\caption{Evolution of the total, elastic, plastic,  viscoelastic  and surface energies for the 1D traction experiment with backtracking, $\tau=1$, $\beta_1=0.01$.}
\label{first}
   \end{minipage}  
 \begin{minipage}[c]{.46\linewidth}
\hspace*{-10mm}
      \includegraphics[width=8cm]{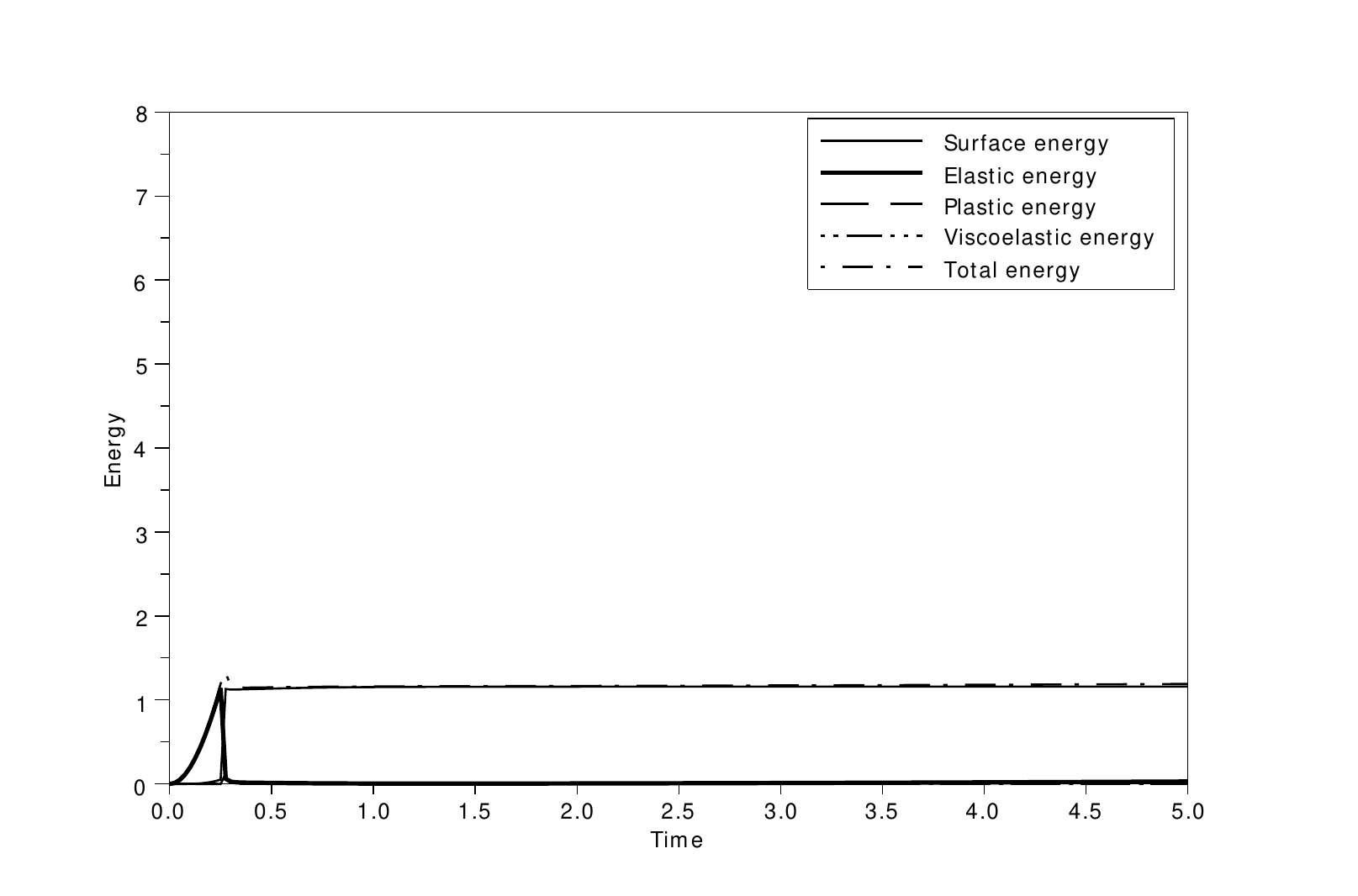}
\caption{Evolution of the total, elastic, plastic,  viscoelastic  and surface energies for the 1D traction experiment with backtracking, $\tau=1.5$, $\beta_1=0.0001$.}
\label{second}
   \end{minipage}  
 \begin{minipage}[c]{.46\linewidth}
\hspace*{-10mm}     
 \includegraphics[width=8cm]{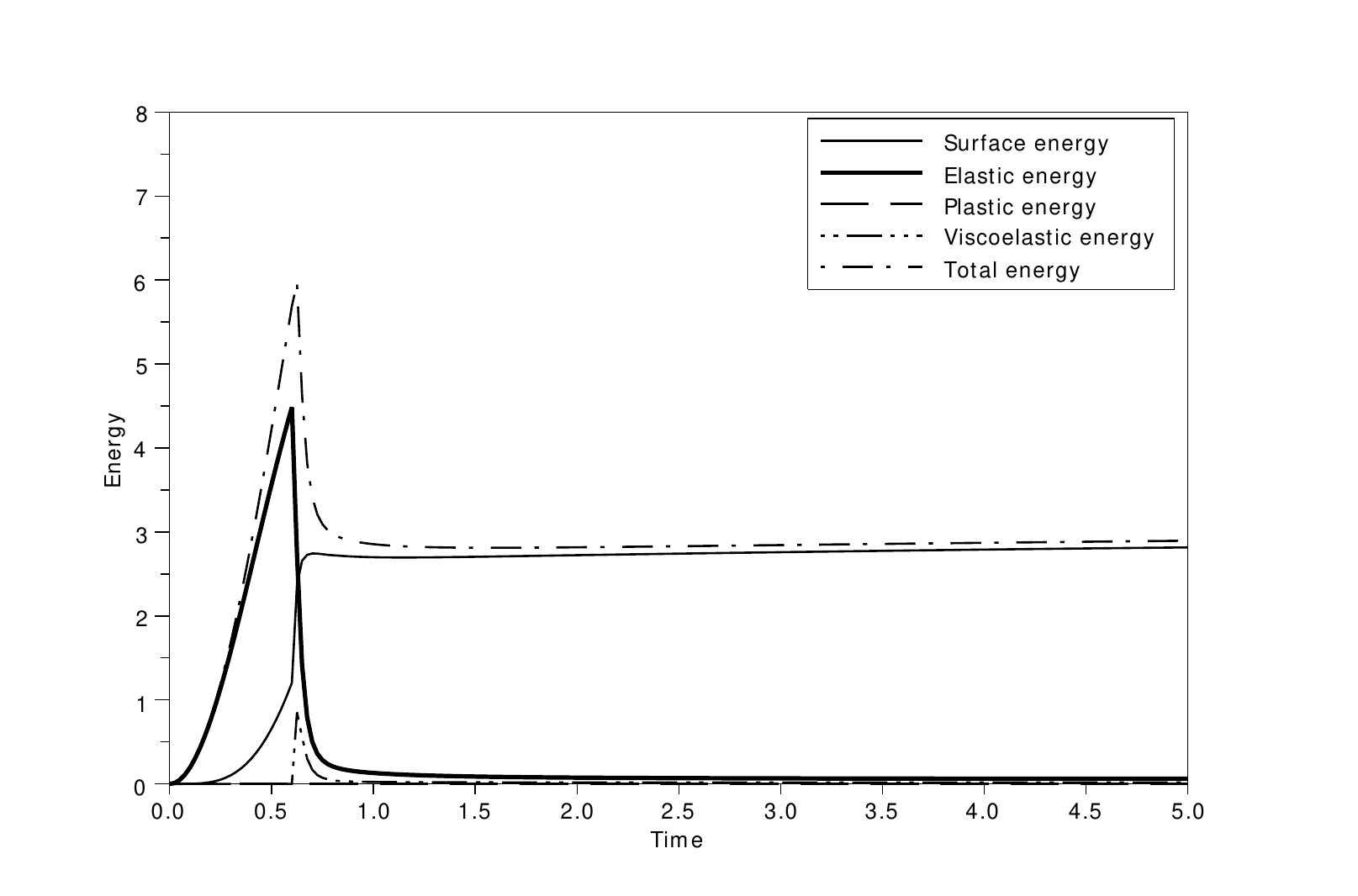}
\caption{Evolution of the total, elastic, plastic,  viscoelastic  and surface energies for the 1D traction experiment with backtracking, $\tau=5$, $\beta_1=0.01$.}
\label{third}
   \end{minipage} \hfill  
   \begin{minipage}[c]{.46\linewidth}
\hspace*{-10mm}     
 \includegraphics[width=8cm]{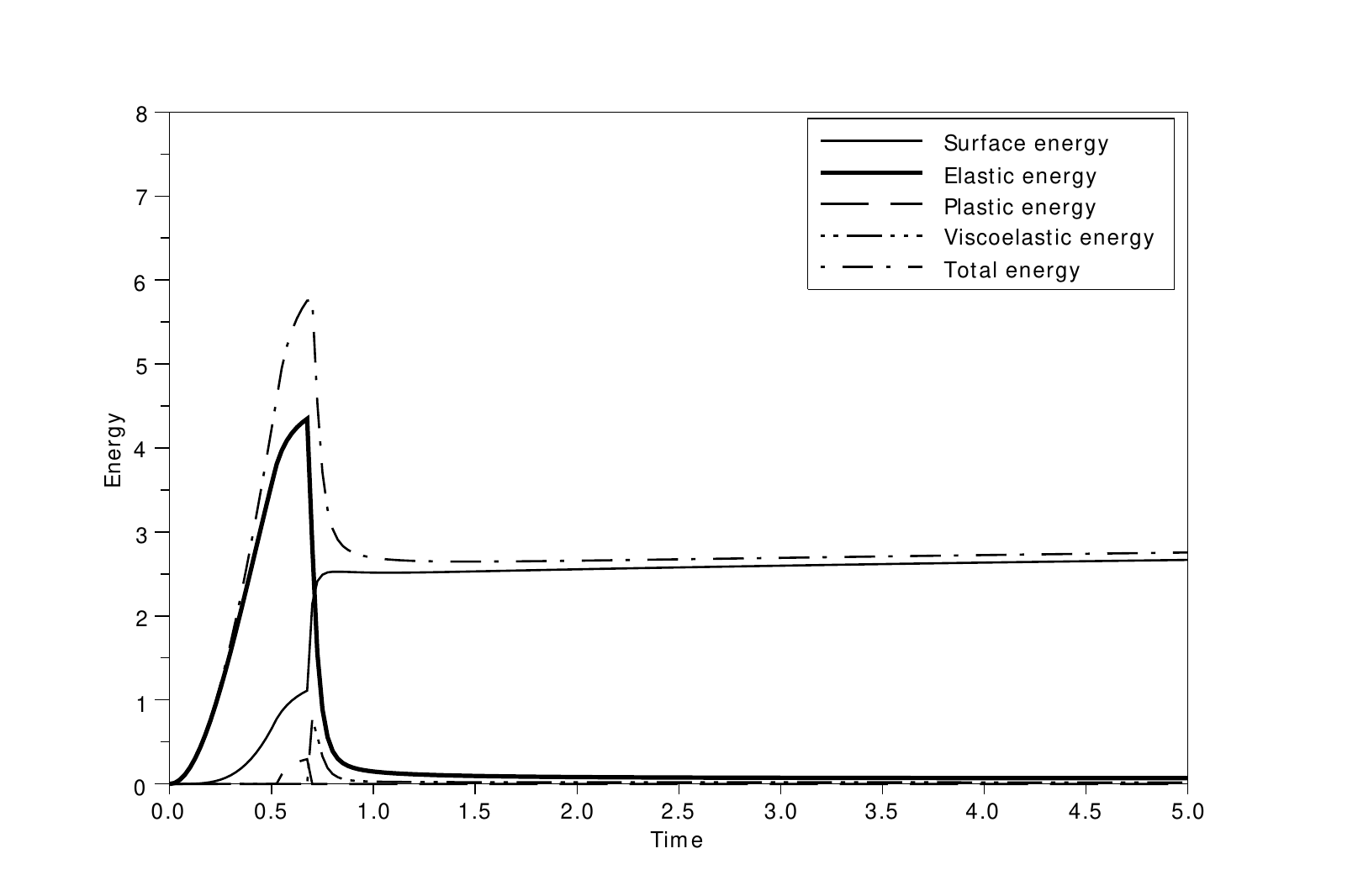}
\caption{Evolution of the total, elastic, plastic, viscoelastic and surface energies for the 1D traction experiment with backtracking, $\tau=1.5$, $\beta_1=0.01$.}
\label{four}
   \end{minipage} \hfill
\end{figure}
\subsubsection{Model 2 - The elasto-viscoplastic model with fracture}
We cannot exclude complex regimes, however, in all our numerical experiments, we only
observed that after the initial elastic regime, either plastic deformation takes place (Figure~\ref{Figure8}), or a crack may appear (Figure~\ref{Figure11}).
\begin{figure}[htbp!]
   \begin{minipage}[c]{.46\linewidth}
\hspace*{-10mm}     
 \includegraphics[width=8cm]{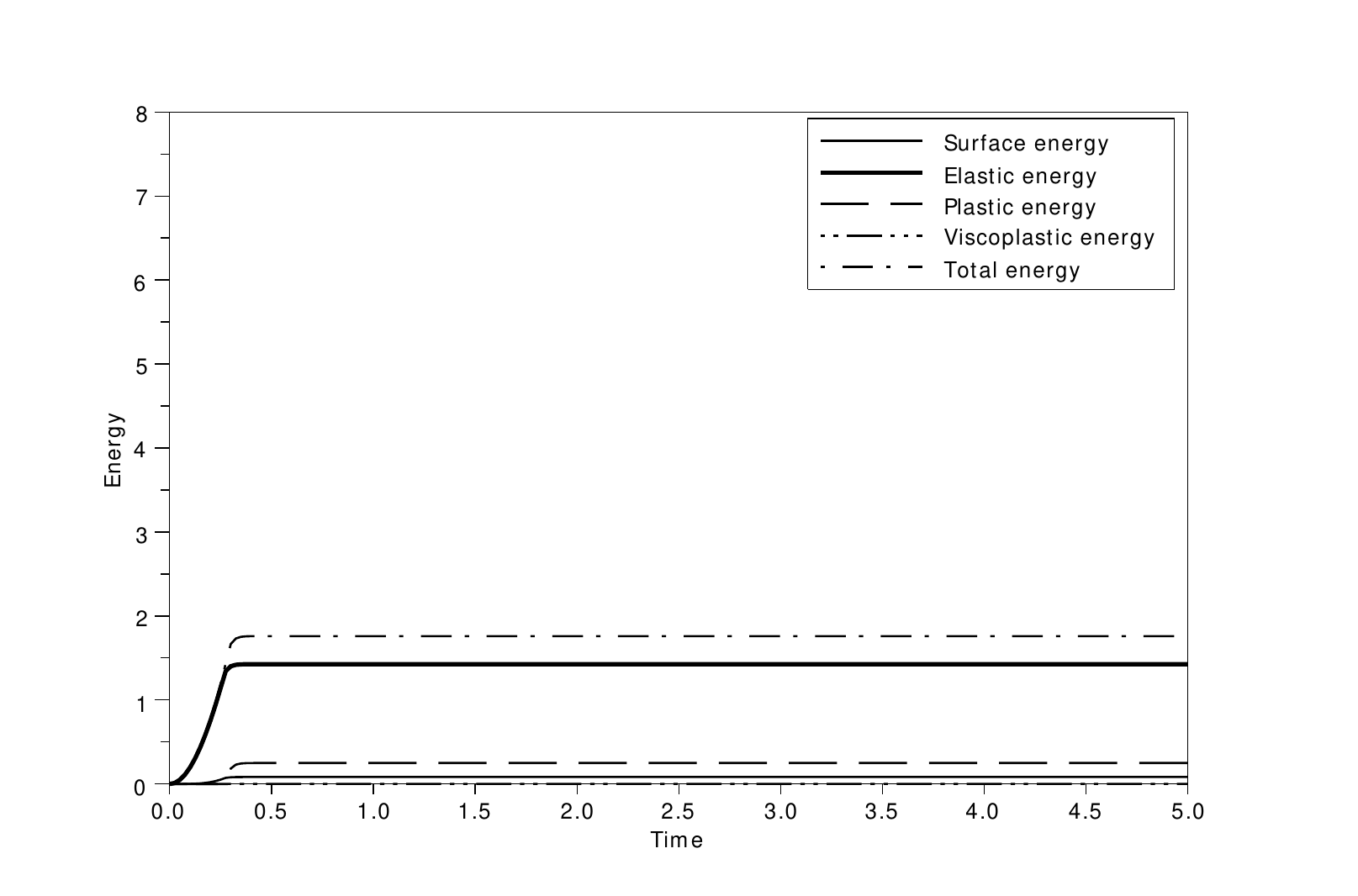}
\caption{Evolution of the total, elastic, plastic, viscoplastic  and surface energies for the 1D traction experiment with backtracking, $\tau=1$, $\beta_2=0.1$.}
\label{Figure8}
   \end{minipage} \hfill
   \begin{minipage}[c]{.46\linewidth}
\hspace*{-10mm}
      \includegraphics[width=8cm]{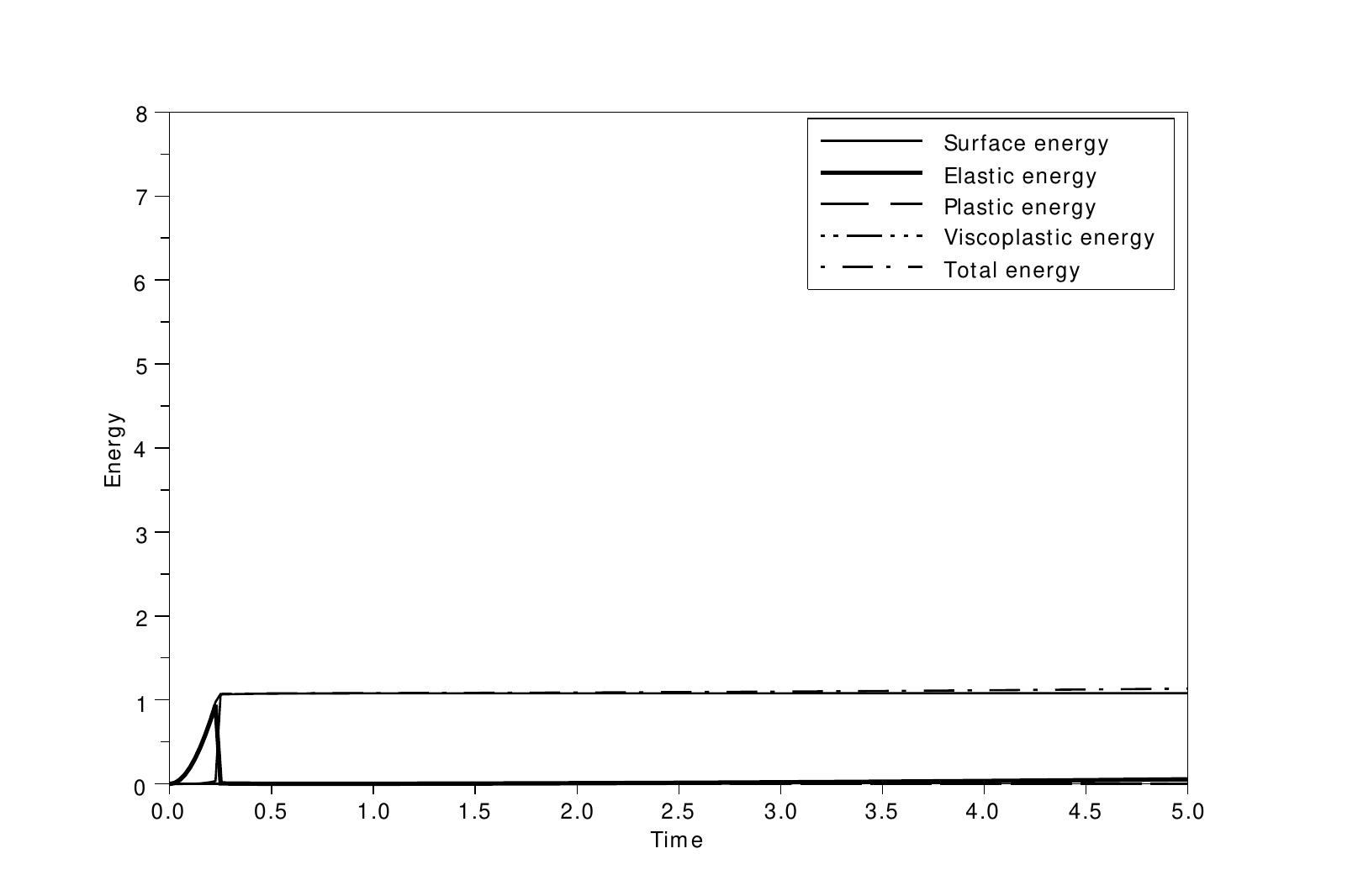}
\caption{Evolution of the total, elastic, plastic, viscoplastic  and surface energies for the 1D traction experiment with backtracking, $\tau=1$, $\beta_2=1$.}
\label{Figure11}
   \end{minipage}
\end{figure}
\subsubsection{Model 3- Elasto-plastic model with linear kinematic hardening and fracture.}\label{Hardening3}
We choose the hardening parameter $k=0.5$. 
For  $\tau=1$, the Figure~\ref{Figure4} shows the elastic behavior with fracture.
For $\tau=0.7$, the medium firstly plastifies and then cracks as depicted in 
Figure~\ref{Figure12}. 
Indeed, kinematic hardening allows the translation of the yield surface and thus the elastic energy can increase after the plastification, so that cracks can appear. 
\begin{figure}[htbp!]
   \begin{minipage}[c]{.46\linewidth}
\hspace*{-10mm}     
 \includegraphics[width=8cm]{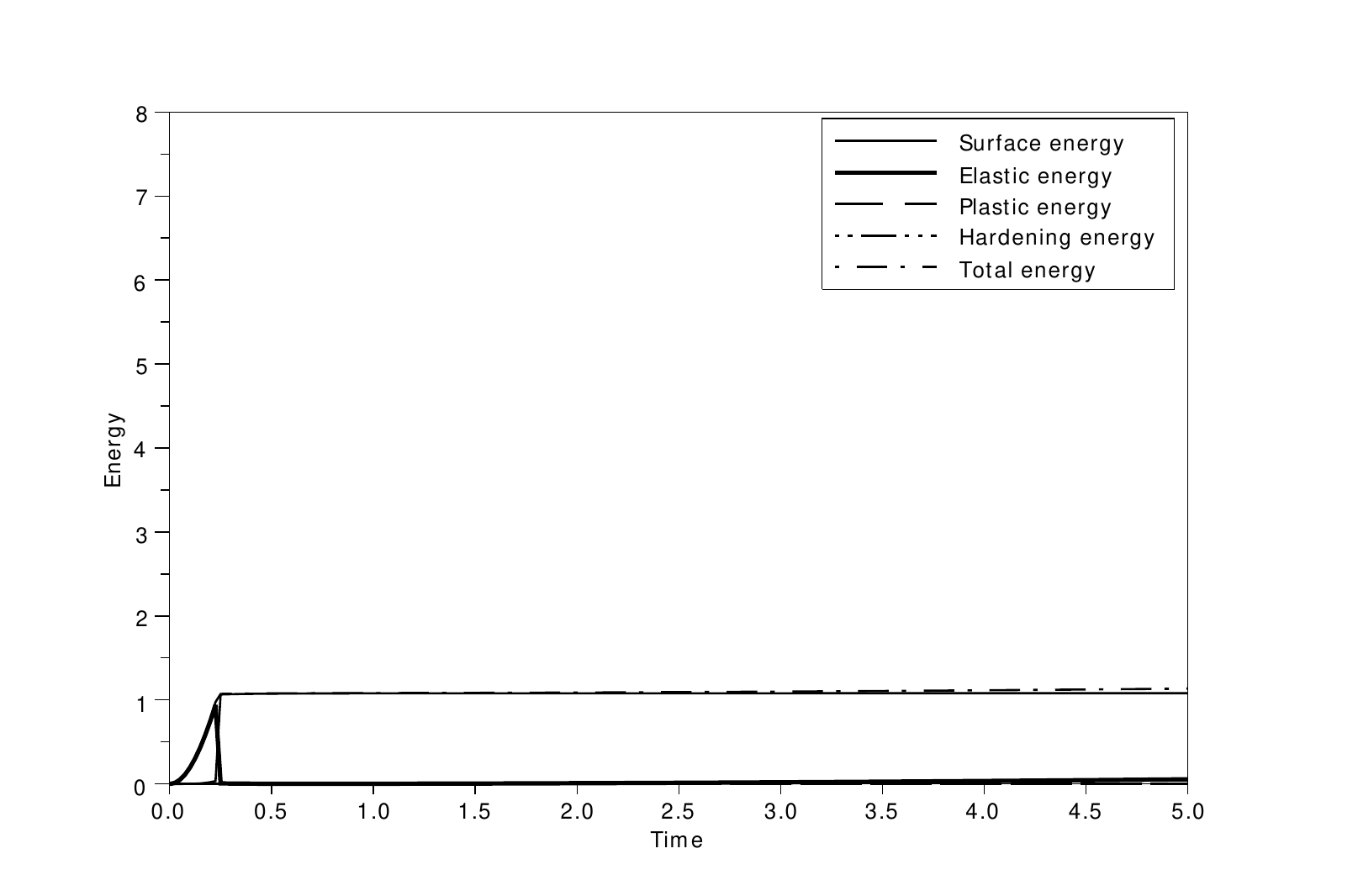}
\caption{Evolution of the total, elastic, plastic, hardening and surface energies for the 1D traction experiment with backtracking, $\tau=1$.}
\label{Figure4}
   \end{minipage}  \hfill
\begin{minipage}[c]{.46\linewidth}
\hspace*{-10mm}     
 \includegraphics[width=8cm]{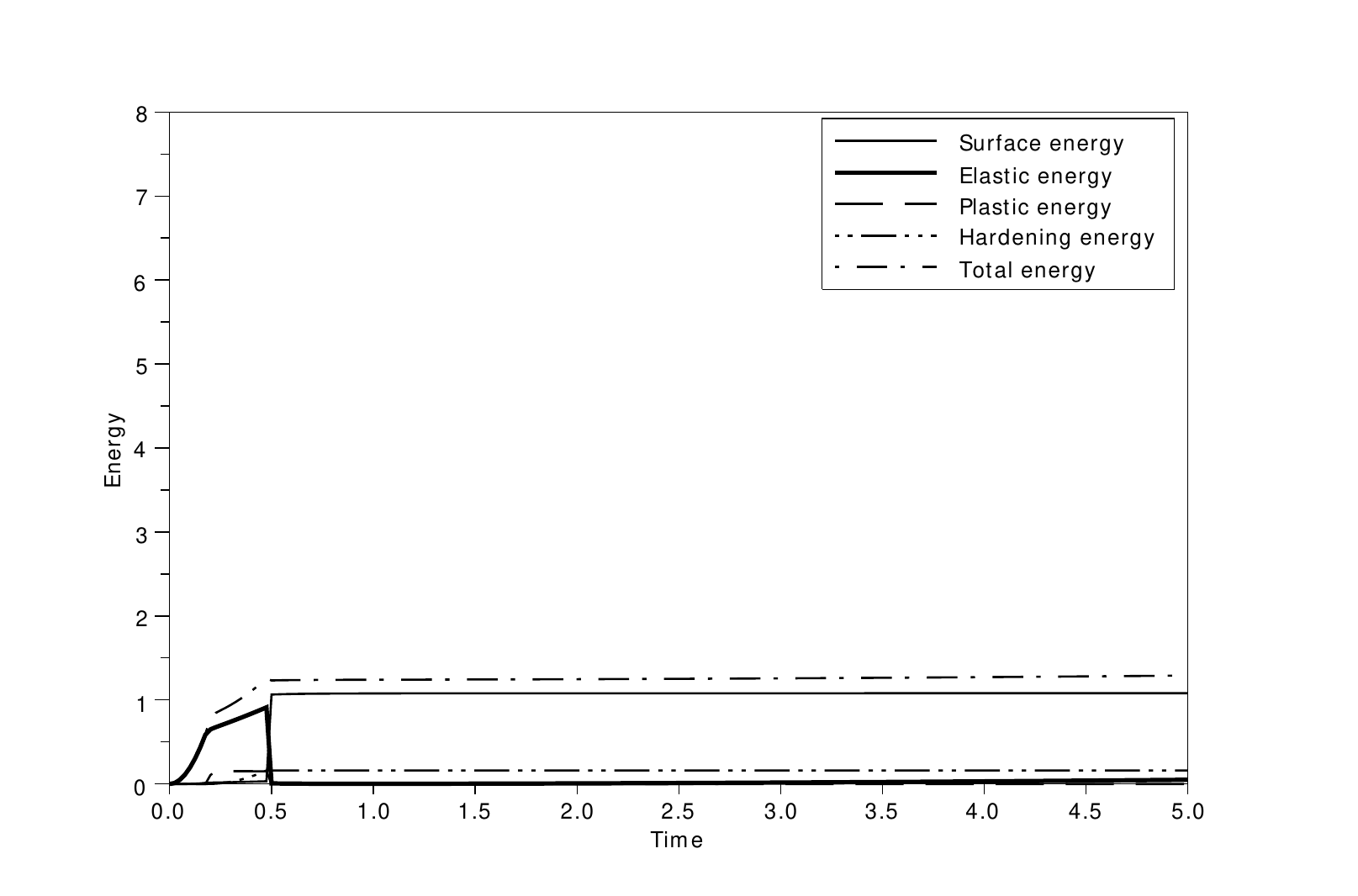}
\caption{Evolution of the total, elastic, plastic, hardening and surface energies for the 1D traction experiment with backtracking, $\tau=0.7$.}
\label{Figure12}
   \end{minipage}
\end{figure}
\subsection{2D-traction numerical experiments-Model 3}
From the numerical 1D-traction experiments of Models 1-3 we conclude that the Models 1 and 3 are those allowing the more complex
evolutions, as all their dissipative mechanisms can be expressed. Those two models have the capacity to plastify the body and then crack during evolution. 
This behavior strongly depends  on the choice of the mechanical parameters. 
Here, we illustrate the behaviour of Model 3 in 2D traction numerical experiments. 

We remark that contrarily to the 1D case, the minimization with respect to p is not explicit.
We compute $p$ with a standard gradient descent method. We consider a beam of length L, and cross section $S=1$ (so that $\Omega=(0,L)\times(0,1))$, which is clamped at $(x,y)=(0,y)$ for $y\in (0,S)$. 
The elastic parameters are the Young modulus $K$ and $\nu$ the Poisson coefficient. The elastic matrix A is defined via Lam\'e's coefficients associated with $K$ and $\mu$. 
For $y\in (0,S)$, we impose at time $t_h^n$  a constant displacement $t_h^nW_0=(t_h^nU_0,0)$ with $U_0>0$ at the right extremity $(x,y)=(L,y)$ of the beam.
We consider the hardening tensor $B$ as a diagonal matrix $B=k\mathbb{I}_2$ with $k>0$ a hardening parameter.
We report numerical experiments with the following parameters: $h=0.1$, $\bigtriangleup x=0.05$, $\varepsilon=0.25$, $K=10$, $k=100$, $\tau=1$, $\nu=0.252$, $U_0=1$.
\begin{figure}[htbp]
   \begin{minipage}[c]{.46\linewidth}
\hspace*{-8mm}     
 \includegraphics[width=6cm]{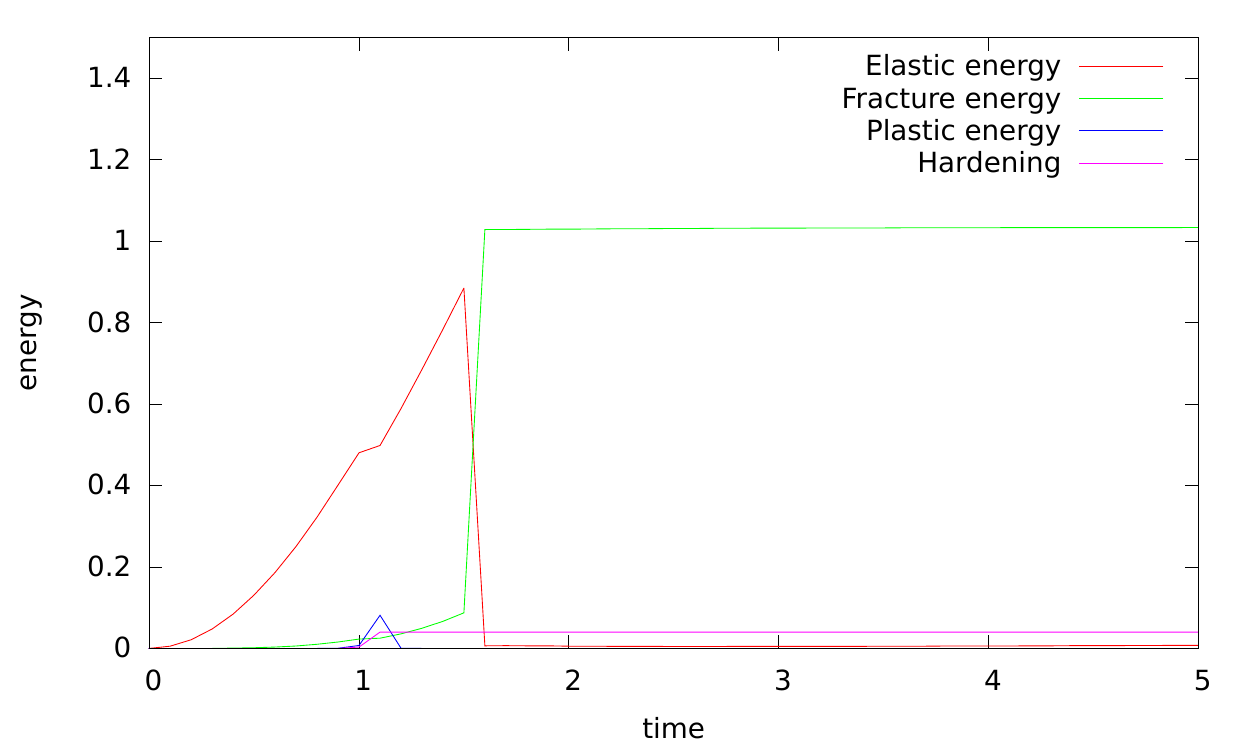}
\caption{Evolution of the elastic, plastic, hardening and surface energies for the 2D traction experiment with backtracking.}
\label{2D1}
 \end{minipage} 
 \hspace*{10mm}
\begin{minipage}[c]{.46\linewidth}
 \includegraphics[width=5cm]{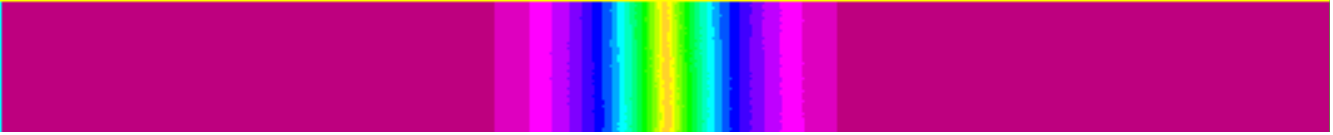}
\caption{Profile of the v(t,.)-fracture approximation at time t=5.}
\label{2D2}
 \vspace*{1cm}
\includegraphics[width=5cm]{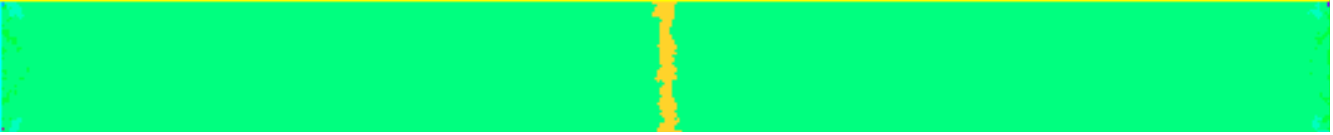}
\caption{Profile of the matrix norm of plastic strain $|p(t,.)|$ at time t=5, ($|p(t,.)|=0.008$ (yellow), $|p(t,.)|=0.009$ (green)).}
\label{2D3}
\end{minipage}
\end{figure}
The evolution of Model 3 shows that with this choice of parameters the material is deformed plastically and then cracks. 
We reproduce qualitatively the same behavior as that the of 1D traction experiment of \ref{Hardening3} (see Figure~\ref{2D1}). 
In Figure~\ref{2D2}, the yellow zone represents the cracked zone.
The magenta zone represents the crack-free zone where $v\sim1$ (see also Figure~\ref{2D3} for plastic strain p). 
\subsection{Numerical simulation of the Peltzer and Tapponnier plasticine experiment}
Using Model 3, we reproduce numerically  the first stages of the plasticine experiment, see Figure~\ref{rupture1}. 
This experiment is meant to model the action of India (as an indenter) on the Tibetan Plateau.

We consider a square domain $\Omega=(0,1)\times (0,1)$, that  represents the layer of plasticine, see Figure~\ref{figomega}.

At time $t_h^n$, the indenter is modeled by a Dirichlet boundary condition $u=t_h ^nU_0=(0,t_h^nU_0)$ with $U_0>0$ on $\partial\Omega_3$.
We set $u=0$ on $\partial\Omega_6$ and $u.\vec{n}=0$ on $\partial\Omega_1$.
We use following parameters:
$h=0.05$, $\bigtriangleup x=0.017$, $\varepsilon=0.15$, $K=100$, $k=100$, $\tau=1$, $\nu=0.252$, $U_0=1$.
In Figure~\ref{2D4} we show the fracture profile at $t=2$, which is in good agreement with the plasticine experiment. In Figure~\ref{2D5},
we observe that Model 3 deforms plastically the layer of plasticine and then cracks. Figure~\ref{2D7} indicates the regions of plastic deformation at time $t=2$.
\begin{figure}[htbp!]
   \begin{minipage}[c]{.44\linewidth}
 \includegraphics[width=5cm]{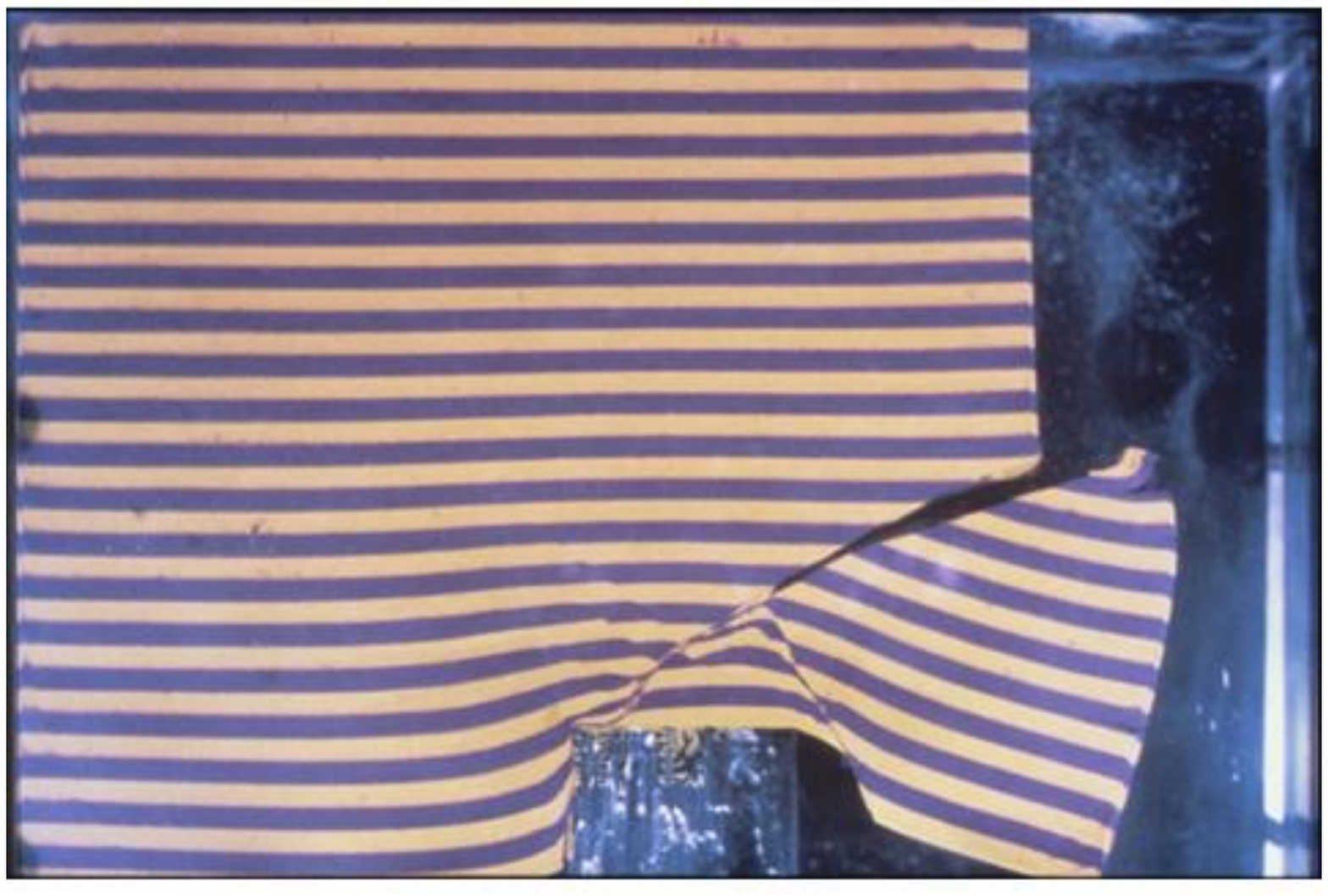}
\caption{Tapponnier and Peltzer's indentation experiment.}
\label{rupture1}
\vspace{0.5cm}
 \includegraphics[width=5cm]{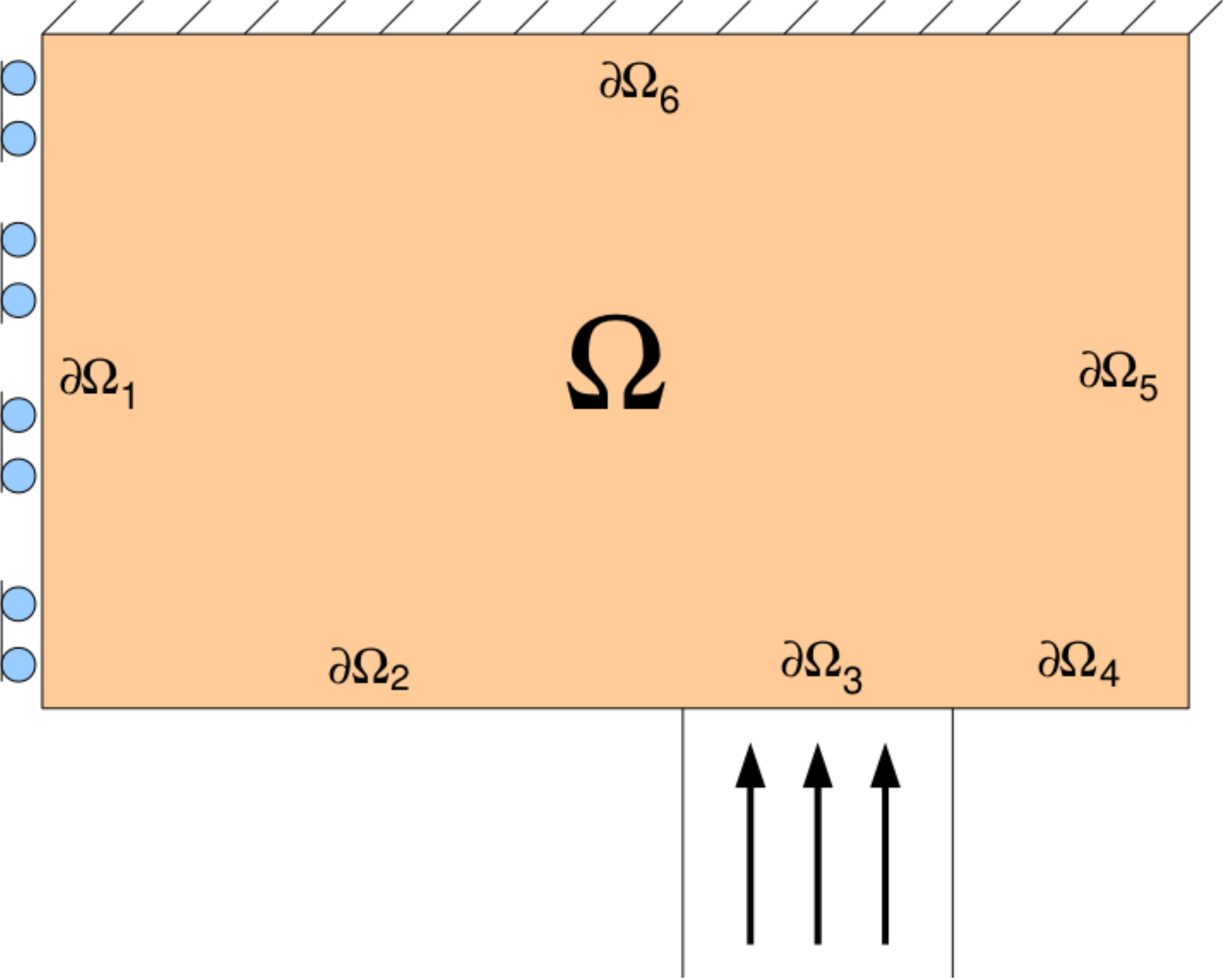}
\caption{Domain $\Omega$ with boundary partition.}  
\label{figomega}
  \end{minipage} \hfill
   \begin{minipage}[c]{.44\linewidth}
      \includegraphics[width=5cm]{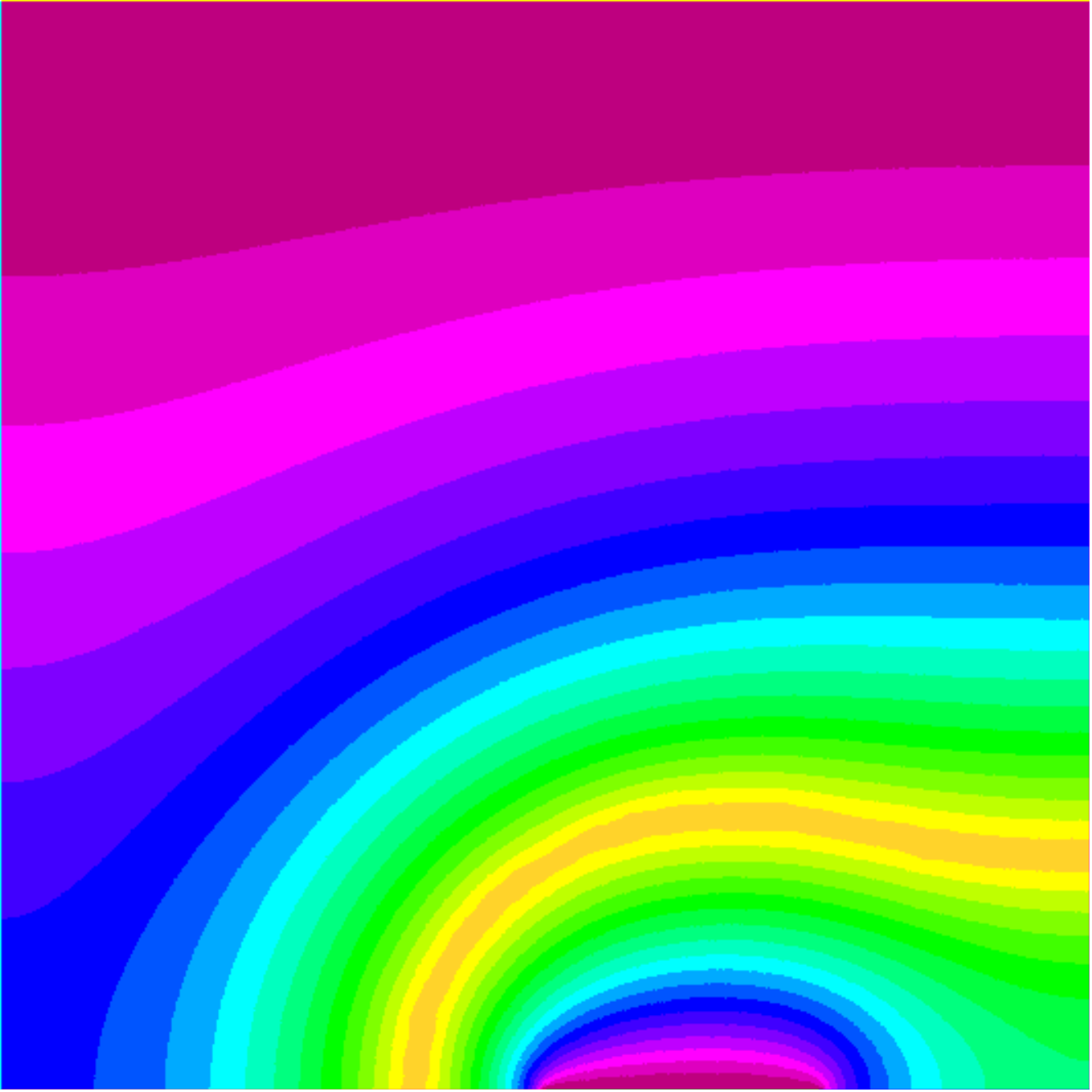}
\caption{Profile of the v(t,.)-fracture approximation at time t=2.}
\label{2D4}
\end{minipage}
\end{figure}
\begin{figure}[htbp!]
\begin{minipage}[c]{.44\linewidth}
\hspace*{-10mm}     
 \includegraphics[width=7cm]{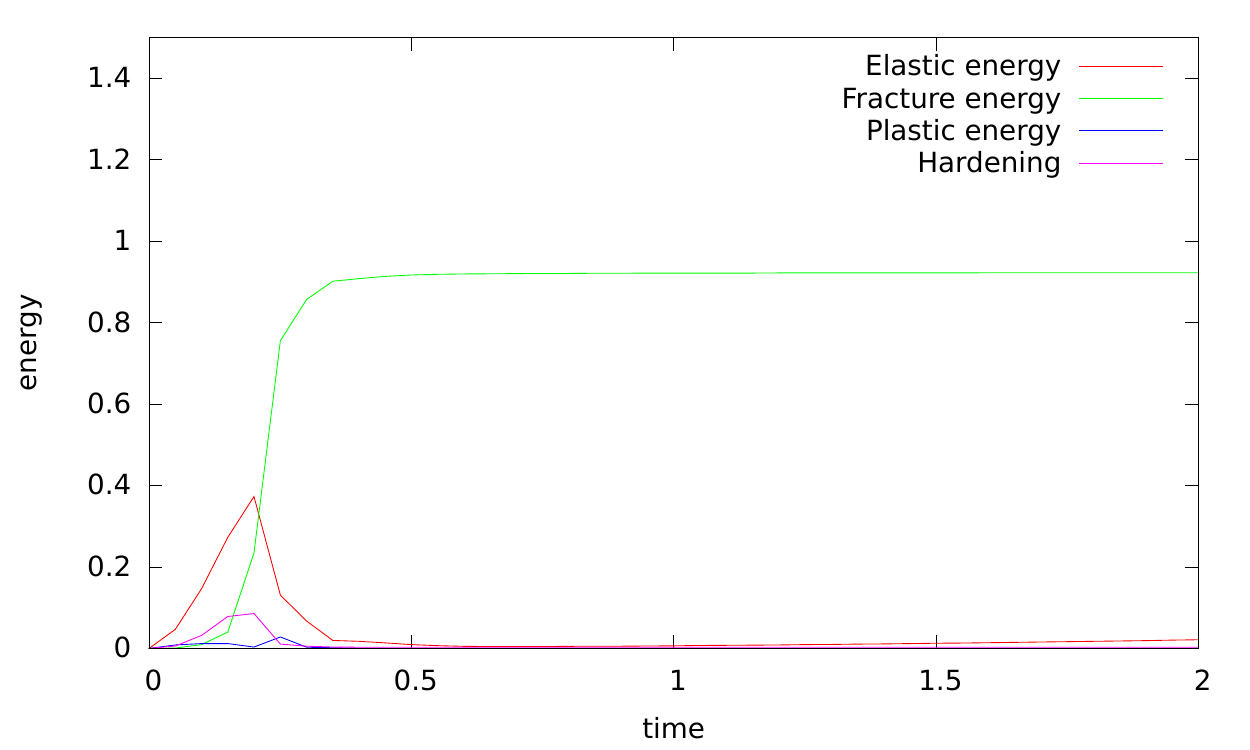}
\caption{Evolution of the elastic, plastic, hardening and surface energies for the 2D plasticine experiment with backtracking.}
\label{2D5}
   \end{minipage} \hfill
   \begin{minipage}[c]{.46\linewidth}
      \includegraphics[width=5cm]{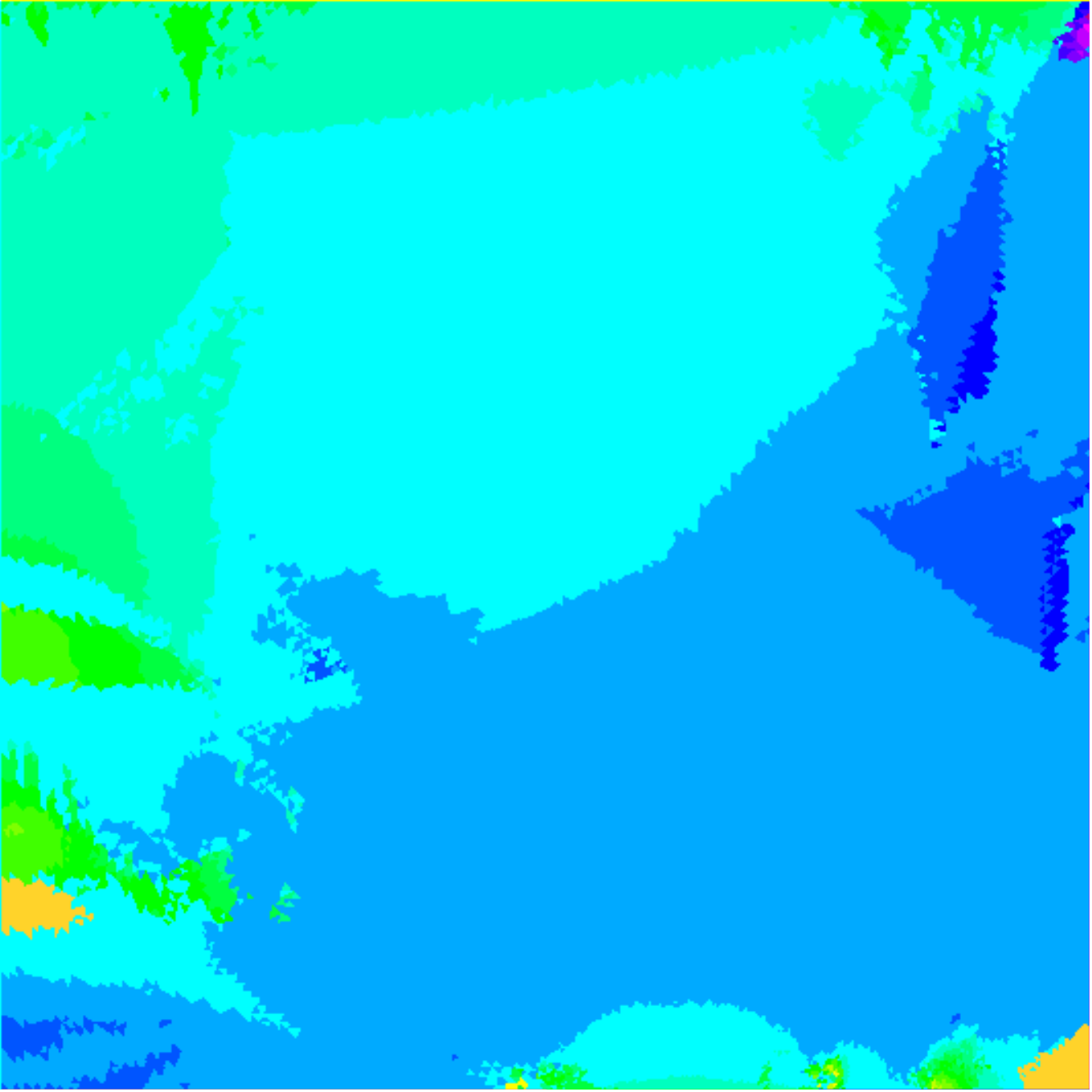}
\caption{Profile of the matrix norm of plastic strain $|p(t,.)|$ at time t=2. ($|p(t,.)|=0.01$ (blue), $|p(t,.)|=0.003$ (green)).}
\label{2D7}
   \end{minipage}
\end{figure}

\section{Conclusion}
In this work, we study 3 models of evolution for materials that can exhibit several dissipation mechanisms: fracture, plasticity, viscous dissipation.
The evolution is defined via a time discretization: at each time step, we seek to minimize a global energy with respect to the variables $(u,p,v)$.
We have reported numerical experiments that show that Models 1 and 3 are most versatile: in particular we can observe evolutions where such materials become
plastic and then crack

In a forthcoming work, we show that we can pass to the limit in Model 1 as time step tends to 0 and give an existence result for a continuous evolution (E1)-(E6).
\section*{Acknowledgements}
The authors wish to express their gratitude to G. Francfort for the  fruitful and enlightening discussions.
This work has been supported by a grant from Labex OSUG@2020 (Investissements d'\,avenir – ANR10 LABX56).

\end{document}